\documentclass[preprint]{imsart}

\usepackage{amsmath,amssymb,amsthm}
\usepackage[comma]{natbib}
\usepackage{pgf}
\usepackage{graphics }
\usepackage{xcolor}

\startlocaldefs
\newcommand{\openR}{\mathbb{R}}

\newcommand{\pp}{\mathbb{P}}

\newcommand{\cov}{{\rm \mathbb{C}ov}}
\newcommand{\half}{\frac{1}{2}}
\newcommand{\ee}{\mathbb{E}\,}
\newcommand{\ii}{\mathcal{I}}

\newcommand{\qq}{\mathbb{Q}}
\newcommand{\R}{\mathbb{R}}

\newcommand{\eps}{\varepsilon}

\newcommand{\tr}{\mbox{tr}}
\newcommand{\dd}{{\rm d}}
\newcommand{\ch}{\mathcal{H}}

\newcommand{\bs}{\mbox{{\boldmath $\Sigma$}}}

\newtheorem{theorem}{Theorem}[section]
\newtheorem{proposition}[theorem]{Proposition}
\newtheorem{lemma}[theorem]{Lemma}
\newtheorem{corollary}[theorem]{Corollary}
\theoremstyle{definition}

\newtheorem{remark}[theorem]{Remark}
\newtheorem{problem}[theorem]{Problem}
\newtheorem{algorithm}[theorem]{Algorithm}

\endlocaldefs

\begin{document}

\begin{frontmatter}

\title{Approximate factor analysis model building via alternating
I-divergence minimization
\protect} 
\runtitle{approximate factor analysis}

\begin{aug}
\author{\fnms{Lorenzo} \snm{Finesso}
        \ead[label=e1]{lorenzo.finesso@isib.cnr.it}} \and
\author{\fnms{Peter} \snm{Spreij} \corref{} \ead[label=e2]{spreij@uva.nl}}
\affiliation{ISIB-CNR and Universiteit van Amsterdam}
\address{Lorenzo Finesso \\ Institute of Biomedical Engineering \\ ISIB-CNR  \\ Padova \\ Italy \\ \printead{e1}}
\address{Peter Spreij \\ Korteweg-de Vries Institute for Mathematics \\ Universiteit van Amsterdam \\ Amsterdam \\
The Netherlands \\ \printead{e2}}
\runauthor{L. Finesso, P. Spreij}
\end{aug}

\begin{abstract}
Given a positive definite covariance matrix $\widehat \Sigma$, we
strive to construct an optimal \emph{approximate} factor analysis
model $HH^\top +D$, with $H$ having a prescribed number of columns
and $D>0$ diagonal. The optimality criterion we minimize is the
I-divergence between the corresponding normal laws. Lifting the
problem into a properly chosen larger space enables us to derive
an alternating minimization algorithm \`a la Csisz\'ar-Tusn\'ady
for the construction of the best approximation. The convergence
properties of the algorithm are studied, with special attention
given to the case where $D$ is singular.
\end{abstract}



\end{frontmatter}

\section{Introduction}\label{section:intro}
\setcounter{equation}{0}
Factor analysis (FA), in its original formulation, \! deals with the linear statistical model
\begin{equation} \label{eq:fam}
Y = HX+\eps,
\end{equation}
where $H$ is a deterministic matrix, $X$ and $\eps$ are independent
random vectors, the first with dimension smaller than $Y$, the
second with independent components. What makes this model attractive
in applied research is the \textit{data reduction} mechanism built
in it. A large number of observed variables $Y$ are explained in
terms of a small number of unobserved (latent) variables $X$
perturbed by the independent noise $\eps$. Under normality
assumptions, which are the rule in the standard theory, all the laws
of the model are specified by covariance matrices. More precisely,
assume that $X$ and $\eps$ are zero mean independent normal vectors
with $\cov(X)=P$ and $\cov(\eps)=D$, where $D$ is diagonal. It
follows from~(\ref{eq:fam}) that $\cov(Y)= HPH^\top +D$. Since in
the present paper, basically only covariances are considered, the
results obtained will also be valid, in a weaker sense, in a non Gaussian environment.

Building a factor analysis model of the observed variables requires the
solution of a difficult algebraic problem. Given $\widehat{\Sigma}$,
the covariance matrix of $Y$, find the triples $(H, P, D)$ such that
$\widehat{\Sigma} = HPH^\top +D$. As it turns out, the right
tools to deal with the construction of an exact FA model come from the theory of stochastic
realization, see~\cite{finessopicci} for an early contribution on
the subject. Due to the structural constraint
on $D$, assumed to be diagonal, the existence and uniqueness of a
FA model are not guaranteed.

In the present paper we strive to construct an optimal \emph{approximate} FA model. The
criterion chosen to evaluate the closeness of covariances is the
I-divergence between the corresponding normal laws. We propose an
algorithm for the construction of the optimal approximation,
inspired by the alternating minimization procedure of~\cite{csiszar}
and~\cite{finessospreij}.

The remainder of the paper is organized as follows. The FA model is
introduced in Section~\ref{section:model} and the approximation
problem is posed and discussed in
Section~\ref{sec:problemstatement}. Section~\ref{section:lift}
recasts the problem as a double minimization in a larger space,
making it amenable to a solution in terms of alternating
minimization. It will be seen that both resulting I-divergence
minimization problems satisfy the so-called Pythagorean rule,
guaranteeing the optimality. In Section~\ref{section:algo}, we
present the alternating minimization algorithm, provide alternative
versions of it, and study its asymptotical properties. We also point
out, in Section~\ref{sec:em}, the relations and differences between
our algorithm and the EM-algorithm for the estimation of the
parameters of a factor analysis model.
Section~\ref{section:SingularD} is dedicated to a constrained
version of the optimization problem (the singular $D$ case) and the
pertinent alternating minimization algorithm. The study of the
singular case also sheds light on the boundary limit points of the
algorithm presented in Section~\ref{section:algo}. In the Appendix
we have collected some known properties on matrix inversion and
I-divergence between normal distributions for easy reference, as
well as most proofs of the technical results.

The present paper is a considerably extended version
of~\cite{fs2007}, moreover providing easier proofs of some of the
results already contained in that reference.

\section{The model}\label{section:model}
\setcounter{equation}{0}

Consider independent random vectors $Z$ and $\eps$, of respective
dimensions $k$ and $n$, both normally distributed with zero mean.
For simplicity $P=\cov(Z)$ is assumed to be invertible. For
any $n \times k$ matrix $L$ let the random vector $Y$, of dimension
$n$, be defined by
\begin{equation}\label{eq:yalt}
Y=LZ+\eps.
\end{equation}
The linear model~(\ref{eq:yalt}), ubiquitous in Statistics, becomes
the standard Factor Analysis (FA) model under the extra constraints
$$k<n, \quad {\rm and} \quad \cov(\eps)=D\ge0, \,\,\, {\rm diagonal.}$$
In many applications one starts with a given, zero mean normal
vector $Y$, and wants to find the parameters $P$, $L$, and $D$ of a
FA model for $Y$. The above constraints impose a special structure
to the covariance of $Y,$
\begin{equation}\label{eq:cov1}
\cov(Y) = LPL^\top + D,
\end{equation}
which is non generic since $k < n$ and $D$ is diagonal, therefore not
all normal vectors $Y$ admit a FA model. To elucidate this, consider
the joint normal vector
\begin{equation}\label{eq:modelL}
V= \begin{pmatrix}
Y \\ Z
\end{pmatrix}
=
\begin{pmatrix}
L & I \\ I & 0
\end{pmatrix}
\begin{pmatrix}
Z \\ \eps
\end{pmatrix},
\end{equation}
whose covariance matrix is given by
\begin{equation} \label{eq:covV}
\cov(V)=\begin{pmatrix} LPL^\top + D & LP  \\
PL^\top & P
\end{pmatrix}.
\end{equation}
The constraints imposed on $\cov(V)$ by the FA model are related to a conditional independence property.
\begin{lemma}\label{lemma:ci}
Let $Y\in \R^n$ be a zero mean normal vector, then
$\cov(Y) = LPL^\top + D$, for some $(L, P, D)$, with $L \in \R^{n \times k}$,
$P>0$, and diagonal $D\ge0$   if and only if
there exists a $k$-dimensional zero mean normal vector $Z$, with $\cov(Z)=P$, such that the components of $Y$ are conditionally independent given $Z$.
\end{lemma}
\begin{proof}
Assume that $\cov(Y)=LPL^\top + D$ and construct a matrix $\Sigma$
as in the right hand side of~(\ref{eq:covV}). Clearly $\Sigma \ge0$,
since $P>0$ and $D\ge0$, and therefore it is a bonafide covariance
matrix, hence there exists a multivariate normal vector $V$ whose
covariance matrix is $\Sigma$. Writing $V^\top =
(Y^\top,Z^\top)^\top$ for this vector, it holds that $\cov(Z)=P$,
moreover $\cov(Y|Z)=D$ (see Equation~(\ref{eq:condcov})). The
conditional independence follows, since $D$ is diagonal by
assumption. For the converse assume there exists a random vector $Z$
as prescribed in the Lemma. Then $\cov(Y|Z)$ is diagonal by the
assumed conditional independence, while $E(Y|Z)=LZ$ for some $L$,
being a linear function of $Z$. We conclude that $\cov(Y) =
\cov(E(Y|Z)) + \cov(Y|Z) = LPL^\top + D$ as requested.
\end{proof}

The above setup is standard in system identification,
see~\cite{finessopicci}. It is often convenient to give an
equivalent reparametrization of model~(\ref{eq:modelL}) as follows.
Let $P=Q^\top Q$, where  $Q$ is a $k \times k$ square root of $P$,
and define $X=Q^{-\top} Z$. Model~(\ref{eq:modelL}) then becomes
\begin{equation*}
V= \begin{pmatrix}
Y \\ Z
\end{pmatrix}
=
\begin{pmatrix}
LQ^\top & I \\ Q^\top & 0
\end{pmatrix}
\begin{pmatrix}
X \\ \eps
\end{pmatrix},
\end{equation*}
where $\cov(X) = I$. The free parameters are now $H=LQ^\top$,
the diagonal $D\ge0$, and the invertible $k \times k$ matrix $Q$. In this paper we will mostly, but not always,
use the latter parametrization, which will be written directly in terms of the newly defined parameters as
\begin{equation}\label{eq:model2}
V= \begin{pmatrix}
Y \\ Z
\end{pmatrix}
=
\begin{pmatrix}
H & I \\ Q^\top & 0
\end{pmatrix}
\begin{pmatrix}
X \\ \eps
\end{pmatrix},
\end{equation}
for which
\begin{equation}\label{eq:cov}
\cov(V)=\begin{pmatrix} HH^\top + D & HQ  \\
(HQ)^\top & Q^\top Q &
\end{pmatrix},
\end{equation}
Note that, with this parametrization,
\begin{equation} \label{eq:y}
Y=HX+ \eps, \qquad \mbox{and} \quad \cov(Y) = HH^\top +D.
\end{equation}
For simplicity, in the first part of the paper, it will be assumed that $H$ has full column rank and $D>0$.

\section{Problem statement}\label{sec:problemstatement}
\setcounter{equation}{0}

Let $Y$ be an $n$ dimensional, normal vector, with zero mean and
$\widehat{\Sigma} = \cov(Y)$ given. As a consequence of
Lemma~\ref{lemma:ci} it is not always possible to find an exact FA
analysis model~(\ref{eq:modelL}), nor
equivalently~(\ref{eq:model2}), for $Y$. As it will be proved
below, one can always find a best approximate FA model. Here
`best' refers to optimizing a given criterion of closeness. In
this paper we opt for minimizing the I-divergence (\emph{a.k.a.}
Kullback-Leibler divergence). Recall that, for given probability
measures $\pp_1$ and $\pp_2$, defined on the same measurable
space, and such that $\pp_1\ll \pp_2$, the I-divergence is defined
as
\begin{equation} \label{eq:KLgen}
\ii(\pp_1||\pp_2)=\mathbb{E}_{\pp_1}\log
\frac{\dd\pp_1}{\dd\pp_2}.
\end{equation}
In the case of normal laws the I-divergence~(\ref{eq:KLgen}) can be
explicitly computed. Let $\nu_1$ and $\nu_2$ be two normal
distributions on $\openR^m$, both with zero mean, and whose
covariance matrices, $\Sigma_1$ and $\Sigma_2$ respectively, are
both non-singular. Then the distributions are equivalent and the
I-divergence $ \ii(\nu_1||\nu_2) $ takes the explicit form, see
Appendix~\ref{gauss},
\begin{equation}\label{eq:divsigma}
\ii(\nu_1||\nu_2)=\half \log
\frac{|\Sigma_2|}{|\Sigma_1|}-\frac{m}{2} +\half {\rm
tr}(\Sigma_2^{-1}\Sigma_1).
\end{equation}
Since, because of zero means, the I-divergence only depends on the
covariance matrices, we usually write $\ii(\Sigma_1||\Sigma_2)$
instead of $\ii(\nu_1||\nu_2)$. Note that $\ii(\Sigma_1||\Sigma_2)$,
computed as in~(\ref{eq:divsigma}), can be considered as a
I-divergence between two positive definite matrices, without
referring to normal distributions. Hence the approximation
Problem~\ref{problem2} below, is meaningful also without normality
assumptions.

The problem of constructing an approximate FA model, i.e. of
approximating a given covariance $\widehat{\Sigma}\in\openR^{n\times
n}$ by $HH^\top+D$, can be cast as the following
\begin{problem}\label{problem2}
Given $\widehat{\Sigma}>0$ of size $n\times n$ and an integer $k<n$,
minimize
\begin{equation}\label{eq:problem}
\ii(\widehat{\Sigma} || HH^\top+D) = \half \log
\frac{|HH^\top+D|}{|\widehat{\Sigma}|}-\frac{n}{2} +\half {\rm
tr}((HH^\top+D)^{-1}\widehat{\Sigma}),
\end{equation}
where the minimum, if it exists, is taken over all diagonal
$D\ge 0$, and $H \in \openR^{n\times k}$.
\end{problem}
Note that $\ii(\widehat{\Sigma} || HH^\top+D) < \infty$ if and only
if $HH^\top+D$ is invertible, which will be a standing assumption in
all that follows.

The first result is that a minimum in Problem~\ref{problem2} indeed
exists. It is formulated as Proposition~\ref{prop:exist} below,
whose proof, requiring results from Section~\ref{section:algo}, is
given in Appendix~\ref{section:tech}.
\begin{proposition}\label{prop:exist}
There exist matrices $H^*\in\openR^{n\times k}$, and nonnegative
diagonal $D^*\in\openR^{n\times n}$, that minimize the I-divergence
in Problem~\ref{problem2}.
\end{proposition}

In a statistical setup, the approximation problem has an equivalent
formulation as an {\em estimation} problem. One then will have a
sequence of {\em idd} observations $Y_1,\ldots,Y_N$, each
distributed according to~(\ref{eq:y}). The matrices $H$ and $D$ are
the unknown parameters that have to be estimated, which can be done
applying the maximum likelihood (ML) method. For big enough $N$, the
sample covariance matrix will be positive definite a.s. under the
assumption that the covariance matrix of the $Y_i$ is positive
definite.  Denote the sample covariance matrix by
$\widehat{\Sigma}$. The computation of the ML estimators of $H$ and
$D$ is equivalent to solving the minimization
problem~\ref{problem2}. Indeed the normal log likelihood $\ell(H,D)$
with $H$ and $D$ as parameters yields
\[
\ell(H,D)=-\frac{N}{2}\log(2\pi)-\half\log |HH^\top +D|-\half
\mbox{tr}\Big((HH^\top + D)^{-1}\widehat{\Sigma}\Big).
\]
One immediately sees that $\ell(H,D)$ is, up to constants not
depending on $H$ and $D$, equal to $-\ii(\widehat{\Sigma}||HH^\top +
D)$. Hence, maximum likelihood estimation completely parallels
I-divergence minimization, only the interpretation is different.

The equations for the maximum likelihood estimators can be found in
e.g. Section 14.3.1 of~\cite{anderson}. In terms of the unknown
parameters $H$ and $D$, with $D$ assumed to be non-singular, they
are
\begin{align}
H & =(\widehat{\Sigma}-HH^\top)D^{-1}H \label{eq:mlh}\\
D & =\Delta(\widehat{\Sigma}-HH^\top).\label{eq:mld}
\end{align}
where $\Delta(M)$, defined for any square $M$, coincides with $M$ on the diagonal and is zero elsewhere. Note that the matrix $HH^\top +D$ obtained by maximum likelihood estimation, is automatically invertible.
Then it can be verified that equation~(\ref{eq:mlh}) is equivalent to
\begin{equation}
H  =\widehat{\Sigma}(HH^\top+D)^{-1}H, \label{eq:mlh2}
\end{equation}
which is also meaningful, when $D$ is not invertible.

The maximum likelihood equations~(\ref{eq:mlh}) and (\ref{eq:mld})
for the alternative parametrization, as induced by~(\ref{eq:modelL}),
take the form
\begin{align}
L & = (\widehat{\Sigma}-LPL^\top)D^{-1}L \label{eq:mll} \\
D & = \Delta(\widehat{\Sigma}-LPL^\top),
\end{align}
with (\ref{eq:mll}) equivalent to
\begin{equation}\label{eq:mll2}
L = \widehat{\Sigma}(LPL^\top+D)^{-1}L.
\end{equation}

It is clear that the system of equations~(\ref{eq:mlh}),
(\ref{eq:mld}) does not have an explicit solution.
For this reason numerical algorithms have been devised, among others
an adapted version of the EM algorithm, see~\cite{rubinthayer1982}. In
the present paper we consider an alternative approach and, in Section~\ref{section:algo},
we compare the ensuing algorithm with the EM.

In~\cite{finessospreij} we considered an approximate nonnegative
matrix factorization problem, where the objective function was also
of I-divergence type. In that case, a relaxation technique lifted
the original minimization to a double minimization in a higher
dimensional space and led naturally to an alternating minimization
algorithm. A similar approach, containing the core of the present
paper, will be followed below.

\section{Lifted version of the problem}\label{section:lift}
\setcounter{equation}{0}

In this section we recast Problem~\ref{problem2} in a higher
dimensional space, making it amenable to solution via two partial
minimizations. Later on this approach will lead to an alternating
minimization algorithm.

First we introduce two relevant classes of normal distributions. All
random vectors are supposed to be zero mean and normal, therefore
their laws are completely specified by covariance matrices. Consider
the set $\bs$ comprising all the $(n+k)$-dimensional covariance
matrices. An element $\Sigma \in \bs$ can always be decomposed as
\begin{equation}\label{eq:dec}
\Sigma=\begin{pmatrix} \Sigma_{11} & \Sigma_{12} \\ \Sigma_{21} &
\Sigma_{22}
\end{pmatrix},
\end{equation}
where $\Sigma_{11}$ and $\Sigma_{22}$ are square, of respective
sizes $n$ and $k$. Two subsets of $\bs$ will play a major role in
what follows. The subset $\bs_0$ of $\bs$, contains the covariances
that can be written as in (\ref{eq:dec}), with
$\Sigma_{11}=\widehat{\Sigma}$, a given matrix, i.e.
\begin{equation}\nonumber
\bs_0=\{\Sigma\in\bs: \Sigma_{11}=\widehat{\Sigma}\}.
\end{equation}
Elements of $\bs_0$ will often be denoted by $\Sigma_0$. Also of interest is the subset $\bs_1$ of $\bs$
whose elements are covariances for which the decomposition~(\ref{eq:dec}) takes the
special form
\begin{equation}\label{eq:dec1}
\Sigma=\begin{pmatrix} HH^\top + D & HQ \\
(HQ)^\top & Q^\top Q
\end{pmatrix},
\end{equation}
for certain matrices $H, D, Q$ with $D$ diagonal, i.e.
\begin{equation}\nonumber
\bs_1=\{\Sigma\in\bs: \exists H,D,Q :\Sigma_{11}=HH^\top + D,
\Sigma_{12}=HQ, \Sigma_{22}=Q^\top Q \}.
\end{equation}
Elements of $\bs_1$ will be often denoted by $\Sigma(H,D,Q)$ or by $\Sigma_1$.

\medskip
In the present section we study the lifted
\begin{problem}\label{liftedproblem}
\[
\min_{\Sigma_0\in\bs_0,\Sigma_1\in\bs_1}\ii(\Sigma_0||\Sigma_1)
\]
\end{problem} \noindent
viewing it as a double minimization over the variables $\Sigma_0$
and $\Sigma_1$. Problem~\ref{liftedproblem} and
Problem~\ref{problem2} are related by the following proposition,
whose proof is deferred to Appendix~\ref{section:tech}.
\begin{proposition}\label{prop:pqq}
Let $\widehat{\Sigma}$ be given. It holds that
\begin{equation*}
\min_{H,D}\, \ii(\widehat{\Sigma}\,||\,HH^\top +
D)=\min_{\Sigma_0\in\bs_0,\Sigma_1\in\bs_1}\ii(\Sigma_0||\Sigma_1).
\end{equation*}
\end{proposition}

\subsection{Partial minimization problems}\label{section:pms}
The first partial minimization, required for the solution of Problem~\ref{liftedproblem}, is as follows.
\begin{problem}\label{problem:pm1}
Given a strictly positive definite covariance matrix $\Sigma\in
\bs$, find
$$\min_{\Sigma_0 \in\bs_0} \, \ii(\Sigma_0||\Sigma).$$
\end{problem}
\noindent The unique solution
to this problem can be computed analytically.
\begin{proposition}\label{prop:pm1}
The unique minimizer $\Sigma^*$ of Problem~\ref{problem:pm1} is given by
\begin{equation}\nonumber
\Sigma^*=\begin{pmatrix} \widehat{\Sigma} &
\widehat{\Sigma}\Sigma_{11}^{-1}\Sigma_{12}
\\
\Sigma_{21}\Sigma_{11}^{-1}\widehat{\Sigma} &
\Sigma_{22}-\Sigma_{21}\Sigma_{11}^{-1}
(\Sigma_{11}-\widehat{\Sigma})\Sigma_{11}^{-1}\Sigma_{12}
\end{pmatrix} >0.
\end{equation}
Moreover
\begin{equation}\label{eq:011}
\ii(\Sigma^*||\Sigma)=\ii(\widehat{\Sigma}||\Sigma_{11}),
\end{equation}
and the Pythagorean rule
\begin{equation}\label{eq:pyth1}
\ii(\Sigma_0||\Sigma)=\ii(\Sigma_0||\Sigma^*)+\ii(\Sigma^*||\Sigma)
\end{equation}
holds for any strictly positive $\Sigma_0\in \bs_0$.
\end{proposition}
\begin{proof} See Appendix~\ref{section:tech}.\end{proof}
\begin{remark}\label{remark:inv}
Using the decomposition of Lemma~\ref{lemma:matrix}, one can easily
compute the inverse of the matrix $\Sigma^*$ of
Proposition~\ref{prop:pm1} and verify that $(\Sigma^*)^{-1}$ differs
from $\Sigma^{-1}$ only in the upper left block. Moreover, in terms of  $L^2$-norms (the $L^2$-norm of a matrix $M$ is $||M||=({\rm tr}(M^\top M))^{1/2}$) we have for the approximation of the inverse the identity $||\Sigma^{-1}-(\Sigma^*)^{-1}||=||\Sigma_{11}^{-1}-\widehat{\Sigma}^{-1}||$.
\end{remark}
Next we turn to the second partial minimization
\begin{problem}\label{problem:pm2}
Given a strictly positive definite covariance matrix $\Sigma\in
\bs$, find
$$\min_{\Sigma_1 \in\bs_1} \, \ii(\Sigma||\Sigma_1).$$
\end{problem}
A solution to this problem is given explicitly in the proposition
below. To state the result we introduce the following notation: for
any nonnegative definite $P$ denote by $P^{1/2}$ any matrix
satisfying $P^{{1/2}^\top} P^{1/2}=P$, and by $P^{-1/2}$ its
inverse, if it exists. Furthermore we put $\tilde{\Sigma}_{11}=\Sigma_{11}-\Sigma_{12}\Sigma_{22}^{-1}\Sigma_{21}$.
\begin{proposition}\label{prop:pm2}
A minimizer $\Sigma(H^*,D^*,Q^*)$ of Problem~\ref{problem:pm2} is given by
\begin{align*}
Q^* & =\Sigma_{22}^{1/2}\\
H^* & = \Sigma_{12}\Sigma_{22}^{-1/2} \\
D^* & =
\Delta(\Sigma_{11}-\Sigma_{12}\Sigma_{22}^{-1}\Sigma_{21}),
\end{align*}
corresponding to the minimizing matrix
\[
\Sigma^*=\Sigma(H^*,D^*,Q^*)=\begin{pmatrix}
\Sigma_{12}\Sigma_{22}^{-1}\Sigma_{21}+\Delta(\Sigma_{11}-\Sigma_{12}\Sigma_{22}^{-1}\Sigma_{21})
& \Sigma_{12} \\
\Sigma_{21} & \Sigma_{22}
\end{pmatrix}.
\]
Moreover, $\ii(\Sigma||\Sigma^*)=\ii(\tilde{\Sigma}_{11}||\Delta(\tilde{\Sigma}_{11}))$ and the Pythagorean rule
\begin{equation}\label{eq:p2}
\ii(\Sigma||\Sigma_1)=\ii(\Sigma||\Sigma^*)+\ii(\Sigma^*||\Sigma_1)
\end{equation}
holds for any $\Sigma_1=\Sigma(H,D,Q)\in\bs_1$.
\end{proposition}

\begin{proof} See Appendix~\ref{section:tech}.
\end{proof}

Note that this problem cannot have a unique
solution in terms of the matrices $H$ and $Q$. Indeed, if $U$ is a
unitary $k\times k$ matrix and $H'=HU$, $Q'=U^\top Q$, then
$H'H'^\top=HH^\top$, $Q'^\top Q'=Q^\top Q$ and $H'Q'=HQ$.
Nevertheless, the optimal matrices $HH^\top$, $HQ$ and $Q^\top Q$
are unique, as it can be easily checked using the expressions in Proposition~\ref{prop:pm2}.
\begin{remark}\label{remark:hhd}
Note that, since $\Sigma$ is supposed to be strictly positive, $\Sigma_{11}-\Sigma_{12}\Sigma_{22}^{-1}\Sigma_{21}>0$.
It follows that $D^*= \Delta(\Sigma_{11}-\Sigma_{12}\Sigma_{22}^{-1}\Sigma_{21})$ is strictly positive.
\end{remark}

\begin{remark}\label{remark:l2d}
The matrix $\Sigma^*$ in Proposition~\ref{prop:pm2} differs from
$\Sigma$ only in the upper left block and in terms of $L^2$-norms we have the identity $||\Sigma-\Sigma^*||=||\tilde{\Sigma}_{11}-\Delta(\tilde{\Sigma}_{11})||$,
compare with Remark~\ref{remark:inv}.
\end{remark}

We close this section by considering a constrained version of the second
partial minimization Problem~\ref{problem:pm2}. The constraint that we impose is $Q=Q_0$, where
$Q_0$ is fixed or, slightly more general, with $P_0:=Q_0^\top Q_0$ fixed. The matrices $H$ and $D$ remain free. For clarity we state this as
\begin{problem}\label{problem:pm2c}
Given strictly positive covariances $\Sigma \in \bs$ and $P_0 \in \openR^{k \times k}$, and letting
$Q_0$ be any matrix satisfying $P_0 = Q_0^\top Q_0$, find
$$\min_{\Sigma(H, D, Q_0) \in\bs_1} \, \ii(\Sigma||\Sigma_1).$$
\end{problem}
The solution is given in the next proposition.
\begin{proposition}\label{prop:sigma0}
A solution $\Sigma_{0}^*$ of Problem~\ref{problem:pm2c} is given by
\[
\Sigma_{0}^*=\begin{pmatrix}
\Sigma_{12}\Sigma_{22}^{-1}P_0\Sigma_{22}^{-1}\Sigma_{21} +
\Delta(\Sigma_{11}-\Sigma_{12}\Sigma_{22}^{-1}\Sigma_{21}) &
\Sigma_{12}\Sigma_{22}^{-1}P_0 \\
P_0\Sigma_{22}^{-1}\Sigma_{21} & P_0
\end{pmatrix},
\]
for which $H^*=\Sigma_{12}\Sigma_{22}^{-1}Q_0^\top$ and $D^*$ is
as in Proposition~\ref{prop:pm2}.
\end{proposition}
\begin{proof}
See Appendix~\ref{section:tech}.
\end{proof}
Note that for the constrained problem no Pythagorean rule holds.
However~(\ref{eq:p2}) can be used to compare the optimal
I-divergences of Problem~\ref{problem:pm2} and
Problem~\ref{problem:pm2c}. Since $\Sigma_0^* \in \bs_1$,
applying~(\ref{eq:p2}) one gets
\[
\ii(\Sigma||\Sigma_{0}^*)=
\ii(\Sigma||\Sigma^*)+\ii(\Sigma^*||\Sigma_{0}^*),
\]
hence $\ii(\Sigma||\Sigma_{0}^*)\geq \ii(\Sigma||\Sigma^*)$,
where $\Sigma^*$ is as in Proposition~\ref{prop:pm2}. The quantity $\ii(\Sigma^*||\Sigma_{0}^*)$
is the extra cost incurred solving Problem~\ref{problem:pm2c} instead of Problem~\ref{problem:pm2}.
An elementary computation gives
$$\ii(\Sigma^*||\Sigma_{0}^*)=\ii(\Sigma_{22}||P_0).$$
In fact this is an easy consequence of the relation, similar to Remark~\ref{remark:inv},
\[
(\Sigma_{0}^*)^{-1}-(\Sigma^*)^{-1}=
\begin{pmatrix}
0 & 0 \\
0 & P_0^{-1}-\Sigma_{22}^{-1}
\end{pmatrix}.
\]
We see that the two optimizing matrices in the constrained case (Proposition~\ref{prop:sigma0})
and unconstrained case (Proposition~\ref{prop:pm2})
coincide iff the constraining matrix $P_0$ satisfies $P_0=\Sigma_{22}$.

\section{Alternating minimization algorithm}\label{section:algo}
\setcounter{equation}{0}

In this section, the core of the paper, the two partial
minimizations of Section~\ref{section:lift} are combined into an
alternating minimization algorithm for the solution of
Problem~\ref{problem2}. A number of equivalent formulations of the
updating equations will be presented and their properties discussed.

\subsection{The algorithm} \label{section:subalgo}

We suppose that the given covariance matrix $\widehat{\Sigma}$ is
strictly positive definite. To setup the iterative minimization
algorithm, assign initial values $H_0, D_0, Q_0$ to the
parameters, with $D_0$ diagonal, $Q_0$ invertible and $H_0H_0^\top
+D_0$ invertible. The updating rules are constructed as follows.
Let $H_t, D_t, Q_t$ be the parameters at the $t$-th iteration, and
$\Sigma_{1,t} = \Sigma(H_t,D_t,Q_t)$ the corresponding covariance,
defined as in~(\ref{eq:dec1}). Now solve the two partial
minimizations as illustrated below.

\begin{equation*}
 (H_t, D_t, Q_t) \,\,\xrightarrow[\underset{\Sigma_0 \in
{\mathbf \Sigma_0}}\min \ii(\Sigma_0 || \Sigma_{1,t})]{{\rm
Prop.~\ref{prop:pm1}}} \,\, \Sigma_{0,t} \,\,
\xrightarrow[\underset{\Sigma_1 \in {\mathbf \Sigma_1}}\min
\ii(\Sigma_{0,t} || \Sigma_1)]{{\rm Prop.~\ref{prop:pm2}}} \,\,
(H_{t+1}, D_{t+1}, Q_{t+1}) \, \cdots,
\end{equation*}
where $\Sigma_{0,t}$ denotes the solution of the first minimization with input $\Sigma_{1,t}$.

\noindent To express in a compact form the resulting update
equations, define
\begin{equation}\label{eq:r}
R_t=I- H_t^\top (H_tH_t^\top + D_t)^{-1}H_t + H_t^\top (H_tH_t^\top
+ D_t)^{-1} \widehat{\Sigma} (H_tH_t^\top + D_t)^{-1}H_t.
\end{equation}
Note that, by Remark~\ref{remark:hhd}, $H_tH^\top_t+D_t$ is
actually invertible for all $t$, since both $H_0H^\top_0+D_0$ and
$Q_0$ have been chosen to be invertible. It follows, by
Corollary~\ref{cor:ihht}, that also $I- H_t^\top (H_tH_t^\top +
D_t)^{-1}H_t$, and consequently $R_t$, are strictly positive and
therefore invertible. The update equations resulting from the
cascade of the two minimizations are
\begin{align}
Q_{t+1} & = \Big(Q_t^\top R_t Q_t\Big)^{1/2}, \label{eq:qn} \\
H_{t+1} & = \widehat{\Sigma}(H_tH_t^\top + D_t)^{-1}H_tQ_tQ_{t+1}^{-1}, \label{eq:h1}\\
D_{t+1} & =
\Delta(\widehat{\Sigma}-H_{t+1}H_{t+1}^\top).\label{eq:dn}
\end{align}
Properly choosing the square root in Equation~(\ref{eq:qn}) makes $Q_t$ disappear from the update equations.
This is an attractive feature since only $H_t$ and $D_t$ are needed to construct the approximate FA model
$H_t H_t^\top + D_t$ at the $t$-th step of the algorithm. Observe that
$(Q_t^\top R_t Q_t)^{1/2} = R_t^{1/2} Q_t$, where $R_t^{1/2}$ is a symmetric root of $R_t$, is a possible
root for the right hand side of Equation~(\ref{eq:qn}). Inserting the resulting matrix $Q_{t+1} = R_t^{1/2} Q_t$
into Equation~(\ref{eq:h1}) results in
\begin{algorithm}\label{algo1}
Given  $H_t$, $D_t$ from the $t$-th step, and $R_t$ as in
(\ref{eq:r}), the update equations for a I-divergence minimizing
algorithm are
\begin{align}
H_{t+1}  & = \widehat{\Sigma}(H_tH_t^\top + D_t)^{-1}H_tR_t^{-1/2} \label{eq:h}\\
D_{t+1} & = \Delta(\widehat{\Sigma}-H_{t+1}H_{t+1}^\top).
\label{eq:d}
\end{align}
\end{algorithm}
\noindent Since $R_t$ only depends on $H_t$ and $D_t$, see~(\ref{eq:r}), the parameter
$Q_t$ has been effectively eliminated.

\subsection{Alternative algorithms}
Algorithm~\ref{algo1} has two drawbacks making its implementation
computationally awkward. To update $H_t$ via equation~(\ref{eq:h})
one has to compute, at each step, the square root of the $k \times
k$ matrix $R_t$ and the inverse of the $n \times n$ matrix
$H_tH_t^\top + D_t$. Taking a slightly different approach it is
possible to reorganize the algorithm in order to avoid the
computation of square roots at each step, and to reduce to $k \times
k$ the size of the matrices that need to be inverted.

\medskip \noindent To avoid the computation of square roots at each step there are at least two possible
variants of Algorithm~\ref{algo1}, both involving a reparametrization. The first approach is to use the
alternative pa\-ra\-me\-tri\-za\-tion~(\ref{eq:modelL}) and to write update equations for the parameters $L,D,P$.
Translated in terms of the matrices $L_t:=H_tQ_t^{-\top}$ and $P_t=Q^\top_tQ_t$, Algorithm~\ref{algo1} becomes
\begin{algorithm}\label{algo2}
Given  $L_t$, $P_t$, and $D_t$ from the $t$-th step, the update
equations for a I-divergence minimizing algorithm are
\begin{align}
L_{t+1}  & = \widehat{\Sigma}(L_t P_t L_t^\top + D_t)^{-1}L_t P_t P_{t+1}^{-1}, \label{eq:l}\\
P_{t+1} & = P_t \! - \! P_tL_t^\top (L_tP_tL_t^\top \! + \!
D_t)^{-1} (L_tP_tL_t^\top \! + \! D_t \! - \! \widehat{\Sigma})(L_t P_t L_t^\top \! + \! D_t)^{-1}L_t P_t, \nonumber\\
D_{t+1} & = \Delta(\widehat{\Sigma}-L_{t+1}P_{t+1}L_{t+1}^\top).
\nonumber
\end{align}
\end{algorithm}

\noindent One can run Algorithm~\ref{algo2} for any number $T$ of steps, and then switch back to the $H, D$
parametrization computing $H_T = L_T Q_T^\top$, which requires only the square root at iteration $T$, \emph{i.e.} $P_T =Q_T^\top Q_T$

\medskip\noindent An alternative approach to avoid the square roots at each iteration of Algorithm~\ref{algo1} is to run it for $\ch_t:=H_tH_t^\top$.
\begin{proposition}\label{prop:ch}
Let $H_t$ be as in Algorithm~\ref{algo1}. Pick $\ch_0=H_0H_0^\top$, and $D_0$ such that $\ch_0+D_0$ is invertible.
The update equation for $\ch_{t}$ becomes
\begin{equation}
\ch_{t+1} =
\widehat{\Sigma}(\ch_{t}+D_{t})^{-1}\ch_{t}\big(D_{t}+\widehat{\Sigma}(\ch_t+D_t)^{-1}\ch_t\big)^{-1}\widehat{\Sigma}.\label{eq:hh2}
\end{equation}
\end{proposition}
\begin{proof}
From Equation~(\ref{eq:h}) one immediately gets
\begin{equation}\label{eq:hhr}
\ch_{t+1}= H_{t+1}H_{t+1}^\top =
\widehat{\Sigma}(\ch_t+D_t)^{-1}H_tR_t^{-1}H_t^\top(\ch_t+D_t)^{-1}\widehat{\Sigma}.
\end{equation}
The key step in the proof is an application of the elementary identity
\[
(I+H^\top P H)^{-1}H^\top=H^\top(I+PHH^\top)^{-1},
\]
valid for all $H$ and $P$ of appropriate dimensions for which both
inverses exist. Note that, by Corollary~\ref{cor:ibc}, the two inverses either both exist or both do not exist.
We have already seen that $R_t$ is invertible and of the type $I+HPH^\top$. Following this recipe, we
compute
\begin{align*}
R_t^{-1}H_t^\top & = H_t^\top\big(I-(\ch_t+D_t)^{-1}\ch_t+(\ch_t+D_t)^{-1}\widehat{\Sigma}(\ch_t+D_t)^{-1}\ch_t\big)^{-1} \\
& = H_t^\top\big( (\ch_t+D_t)^{-1}D_t+(\ch_t+D_t)^{-1}\widehat{\Sigma}(\ch_t+D_t)^{-1}\ch_t\big)^{-1} \\
& = H_t^\top \big(D_t +
\widehat{\Sigma}(\ch_t+D_t)^{-1}\ch_t\big)^{-1} (\ch_t+D_t).
\end{align*}
Insertion of this result into~(\ref{eq:hhr}) yields (\ref{eq:hh2}).
\end{proof}
\noindent One can run the update Equation~(\ref{eq:hhr}), for any number $T$ of steps, and then switch back to $H_T$,
taking any $n\times k$ factor of $\ch_T$ \emph{i.e.} solve $\ch_T = H_T H_T^\top $. Since Equation~(\ref{eq:hhr})
transforms $\ch_t$ into $\ch_{t+1}$ preserving the rank, the latter factorization is always possible.

\medskip \noindent It is apparent that the second computational issue we mentioned above, concerning the inversion
of $n\times n$ matrices at each step, affects also Algorithm~\ref{algo2}. The alternative form of the update equations
derived below only requires the inversion of $k \times k$ matrices: a very desirable property since $k$ is usually much
smaller than $n$. Referring to Algorithm~\ref{algo1}, since $D_t$ is invertible, apply Corollary~\ref{cor:inv} to find
\[
(H_tH_t^\top + D_t)^{-1}H_t= D_t^{-1}H_t(I+H_t^\top D_t^{-1}H_t)^{-1}.
\]
The alternative expression for $R_t$ is
\[
R_t=(I+H_t^\top D_t^{-1}H_t)^{-1}+(I+H_t^\top
D_t^{-1}H_t)^{-1}H_t^\top D_t^{-1}\widehat{\Sigma}
D_t^{-1}H_t(I+H_t^\top D_t^{-1})^{-1}.
\]
The update formula~(\ref{eq:h}) can therefore be replaced with
\[
H_{t+1}  = \widehat{\Sigma}D_t^{-1}H_t(I+H_t^\top
D_t^{-1}H_t)^{-1}R_t^{-1/2}.
\]
Similar results can be derived also for Algorithm~\ref{algo2}.

\subsection{Asymptotic properties}
In the portmanteau proposition below we collect the asymptotic
properties of Algorithm~\ref{algo1}, also quantifying the
I-divergence decrease at each step.

\begin{proposition}\label{prop:properties}
For Algorithm~\ref{algo1} the following hold.
\begin{itemize}

\item[(a)] $H_tH_t^\top \leq \widehat{\Sigma}$ for all $t\geq 1$.

\item[(b)] If $D_0>0$ and $\Delta(\widehat\Sigma - D_0)>0$ then $D_t>0$ for all $t\geq 1$.

\item[(c)] The matrices $R_t$ are invertible for all $t\geq 1$.

\item[(d)] If $H_t H_t^\top + D_t = \widehat\Sigma$ \,then $H_{t+1}=H_t, \,D_{t+1}=D_t$.

\item[(e)] Decrease of the objective function:
$$\ii(\widehat{\Sigma}||\widehat\Sigma_t) - \ii(\widehat{\Sigma}||\widehat\Sigma_{t+1})=
\ii(\Sigma_{1,t+1}||\Sigma_{1,t})+\ii(\Sigma_{0,t}||
\Sigma_{0,t+1}),$$ where $\widehat\Sigma_t = H_tH_t^\top + D_t$ is
the $t$-th approximation of $\widehat\Sigma$, and $\Sigma_{0,t},
\Sigma_{1,t}$ were defined in subsection~\ref{section:subalgo}.

\item[(f)] The interior limit points $(H,D)$ of the
algorithm satisfy
\begin{equation} \label{stateqs}
H  = (\widehat{\Sigma}-HH^\top)D^{-1}H, \qquad\qquad D =
\Delta(\widehat{\Sigma}-HH^\top),
\end{equation}
which are the ML equations~(\ref{eq:mlh}) and (\ref{eq:mld}).
If $(H, D)$ is a solution to these equation also $(HU, D)$ is a solution, for any
unitary matrix $U \in \openR^{k \times k}$.

\item[(g)] Limit points  $(\ch,D)$, see~(\ref{eq:hhr}), satisfy
$$\ch=\widehat{\Sigma}(\ch+D)^{-1}\ch, \qquad\qquad D=\Delta(\widehat{\Sigma}-\ch).$$
\end{itemize}
\end{proposition}

\begin{proof}
\noindent(a) This follows from Remark~\ref{remark:hhd} and the construction
of the algorithm as a combination of the two partial minimizations.

\noindent (b) This similarly follows from Remark~\ref{remark:hhd}.

\noindent (c)  Use the identity $I-H_t^\top(H_tH_t^\top +
D_t)^{-1}H_t=(I+H_t^\top D_t^{-1}H_t)^{-1}$ and $\widehat{\Sigma}$
nonnegative definite.

\noindent (d) In this case, Equation~(\ref{eq:r}) shows that
$R_t=I$ and substituting this into the update equations yields the
conclusion.

\noindent (e) As matter of fact, we can express the decrease as a
sum of two I-divergences, since the algorithm is the superposition
of the two partial minimization problems. The results follows from a
concatenation of Proposition~\ref{prop:pm1} and
Proposition~\ref{prop:pm2}.

\noindent (f) We consider Algorithm~\ref{algo2} first. Assume that all
variables converge. Then, from~(\ref{eq:l}), for limit points $L, P, D$
it holds that
\[
L=\widehat{\Sigma}(LPL^\top+ D)^{-1}L,
\]
which coincides with equation~(\ref{eq:mlh}). Let then $Q$ be
a square root of $P$ and $H=LQ^{\top}$. This gives the first of the desired relations. The rest is trivial.

\noindent (g) This follows by inserting the result of (f).
\end{proof}

\noindent In part~(f) of Proposition~\ref{prop:properties} we have
made the assumption that the limit points are interior points.
This assumption does not always hold true, it may happen that a
limit point $(H, D)$ is such that $D$ contains zeros on the
diagonal. We will treat this extensively in
Section~\ref{section:statpoints} in connection with a restricted
optimization problem, in which it is imposed that $D$ has a number
of zeros on the diagonal.

\section{Comparison with the EM algorithm}\label{sec:em}
\setcounter{equation}{0}
\cite{rubinthayer1982} put forward a version of the EM algorithm (see~\cite{dlr})
in the context of estimation for FA models. Their algorithm is as follows.

\begin{algorithm}[EM]\label{em}
\begin{align}
H_{t+1}  & = \widehat\Sigma(H_t H_t^\top + D_t)^{-1}H_t R_t^{-1} \label{EM1}\\
D_{t+1} & = \Delta(\widehat{\Sigma}- H_{t+1} R_t H_{t+1}^\top),
\label{EM2}
\end{align}
where  $R_t= I- H_t^\top (H_t H_t^\top + D_t)^{-1} (H_t H_t^\top +
D_t - \widehat{\Sigma})(H_t H_t^\top + D_t)^{-1} H_t$.
\end{algorithm}
\noindent The EM Algorithm~\ref{em} differs in both
equations from our Algorithm~\ref{algo1}.
It is well known that EM algorithms can be derived as alternating minimizations,
see~\cite{csiszar}, it is therefore interesting to investigate how Algorithm~\ref{em} can
be derived within our framework. Thereto one considers the first partial minimization
problem together with the {\em constrained} second partial
minimization Problem~\ref{problem:pm2c}, the constraint
being $Q=Q_0$, for some  $Q_0$. Later on we will see that the
particular choice of $Q_0$, as long as it is invertible, is
irrelevant. The concatenation of these two problems results in the
EM Algorithm~\ref{em}, as is detailed below.

Starting at $(H_t,D_t,Q_0)$, one performs the first partial
minimization, that results in the matrix
\[
\begin{pmatrix}
\widehat{\Sigma} & \widehat{\Sigma}(H_tH_t+D_t)^{-1}H_tQ_0 \\
Q_0^\top H_t^\top(H_tH_t+D_t)^{-1}\widehat{\Sigma} & Q_0^\top R_t
Q_0
\end{pmatrix}.
\]
Performing now the \emph{constrained} second minimization, according to the results of
Proposition~\ref{prop:sigma0}, one obtains
\begin{align}
H_{t+1} & = \widehat{\Sigma}(H_tH_t^\top +D_t)^{-1}H_t R_t^{-1} \label{eq:h2}\\
D_{t+1} & = \Delta\big(\widehat{\Sigma}-\widehat{\Sigma}(H_tH_t^\top
+D_t)^{-1}H_t R_t ^{-1}  H_t^\top(H_tH_t^\top
+D_t)^{-1}\widehat{\Sigma}\big).\label{eq:d2}
\end{align}
Substitution of~(\ref{eq:h2}) into~(\ref{eq:d2}) yields
\[
D_{t+1}=\Delta(\widehat{\Sigma}-H_{t+1}R_tH_{t+1}^\top).
\]
One sees that the matrix $Q_0$ does not appear in the recursion, just as the matrices
$Q_t$ do not occur in Algorithm~\ref{algo1}.

Both Algorithms~\ref{algo1} and~\ref{em} are the result of two
partial minimization problems. The latter algorithm differs from
ours in that the second partial minimization is \emph{constrained}.
It is therefore reasonable to expect that, from the point of view of
minimizing I-divergence, Algorithm~\ref{algo1} yields a better
performance, although comparisons must take into account that the
initial parameters for the two \emph{species} of the second partial
minimization will in general be different. We will illustrate these considerations by some numerical examples in Section~\ref{section:numerics}.

We also note that for Algorithm~\ref{algo1} it was possible to
identify the update gain at each step, see
Proposition~\ref{prop:properties}(e), resulting from the two
Pythagorean rules. For the EM algorithm a similar formula cannot be
given, because for the constrained second partial minimization
a Pythagorean rule does not hold, see the
discussion after Proposition~\ref{prop:sigma0} in
Section~\ref{section:pms}.

%

\section{Singular $D$}\label{section:SingularD}
\setcounter{equation}{0}

It has been known for a long time, see e.g.~\cite{joreskog}, that
numerical solutions to the ML equations (see
Section~\ref{sec:problemstatement}) often produce a nearly singular
matrix $D$. This motivates the investigation of the stationary
points $(H,D)$ of Algorithm~\ref{algo1} with singular $D$,
\emph{i.e.} with zeros on the diagonal
(Section~\ref{section:statpoints}). Naturally connected to this is
the analysis of the minimization Problem~\ref{problem2} when $D$ is
\emph{constrained}, at the outset, to be singular
(Section~\ref{sec:singularD}), and the investigation of its
consequences for the minimization algorithm of
Proposition~\ref{prop:ch} (Section~\ref{sec:algo0}).

\subsection{Stationary points $(H,D)$ with singular $D$}  \label{section:statpoints}

As mentioned before, already in~\cite{joreskog} it has been observed
that, numerically maximizing the likelihood, one often reaches
matrices $D$ that are nearly singular. This motivates the
investigation of the stationary points $(H,D)$ of
Algorithm~\ref{algo1} for which $D$ is singular, i.e.
\begin{equation}\label{eq:dd1}
D=\begin{pmatrix}
D_1 & 0 \\
0 &  D_2
\end{pmatrix}=\begin{pmatrix}
D_1 & 0 \\
0 &  0
\end{pmatrix},
\end{equation}
where $D_1>0$  has size $n_1\times n_1$ and the lower right zero
block has size $n_2\times n_2$, with $n_1+n_2=n$.

\noindent Accordingly we partition $H\in \R^{n\times k}$ as
\begin{equation}\label{eq:hh12}
H=\begin{pmatrix} H_1 \\H_2
\end{pmatrix},
\end{equation}
where $H_1\in\R^{n_1\times k}$ and $H_2\in\R^{n_2\times k}$. Then
\begin{equation} \label{hhd}
HH^\top + D =
\begin{pmatrix}
H_1H_1^\top + D_1& H_1H_2^\top \\
H_2H_1^\top           & H_2H_2^\top
\end{pmatrix}.
\end{equation}

\noindent We recall that Problem~\ref{problem2} calls for the
minimization, over $H$ and $D$, of the functional
$\ii(\widehat{\Sigma}||HH^\top + D)$, which is finite if and only if
$HH^\top +D$ is strictly positive definite. In view of~(\ref{hhd}),
this happens if and only if $$H_2H_2^\top>0,$$ the standing
assumption of this section. A direct consequence of this assumption
is that $n_2 \le k$.

\noindent The \emph{given} matrix $\widehat{\Sigma}$ will be
similarly decomposed as
\begin{align}\label{eq:ssigma}
\widehat \Sigma = & \begin{pmatrix} \widehat \Sigma_{11} & \widehat \Sigma_{12} \\
\widehat \Sigma_{21} & \widehat \Sigma_{22}.
\end{pmatrix}
\end{align}
\begin{proposition}\label{prop:d2}
If $(H,D)$ is a stationary point of the algorithm, with $D$ as in
(\ref{eq:dd1}), then the given matrix $\widehat{\Sigma}$ is such
that $\widehat{\Sigma}_{22}=H_{2}H_{2}^{\top}$ and
$\widehat{\Sigma}_{12}=H_{1}H_{2}^{\top}$.
\end{proposition}
\begin{proof}  By
Proposition~\ref{prop:properties} $\widehat{\Sigma}-HH^{\top}$ is
nonnegative definite, as is its lower right block
$\widehat{\Sigma}_{22}-H_{2}H_{2}^{\top}$. Since
$D=\Delta(\widehat{\Sigma}-HH^{\top})$ and $D_{2}=0$, we get that
$\Delta(\widehat{\Sigma}_{22}-H_{2}H_{2}^{\top})=0$ and therefore
$\widehat{\Sigma}_{22}=H_{2}H_{2}^{\top}$. We conclude that
\[
\widehat{\Sigma}-HH^{\top} = \begin{pmatrix}
\widehat{\Sigma}_{11}-H_{1}H_{1}^{\top} & \widehat{\Sigma}_{12}-H_{1}H_{2}^{\top} \\
\widehat{\Sigma}_{21}-H_{2}H_{1}^{\top} & 0
\end{pmatrix} \ge 0,
\]
hence $\widehat{\Sigma}_{12}=H_{1}H_{2}^{\top}$.
\end{proof}

\noindent Define
\begin{equation} \label{Htilde}
\widetilde{H}_{1}:=H_{1}(I-H_{2}^{\top}(H_{2}H_{2}^{\top})^{-1}H_{2}).
\end{equation}
Since $I-H_{2}^{\top}(H_{2}H_{2}^{\top})^{-1}H_{2}$ is a projection,
one finds
\begin{equation} \label{HtildeHtilde}
\widetilde{H}_{1}\widetilde{H}_{1}^{\top}=H_{1}(I-H_{2}^{\top}(H_{2}H_{2}^{\top})^{-1}H_{2})H_{1}^{\top}.
\end{equation}
In view of Proposition~\ref{prop:d2} this becomes
\begin{equation} \label{HtildeHtildeS}
\widetilde{H}_{1}\widetilde{H}_{1}^{\top}=H_{1}H_{1}^{\top}-\widehat{\Sigma}_{12}\widehat{\Sigma}_{22}^{-1}\widehat{\Sigma}_{21}.
\end{equation}
Finally we need
\begin{align} \label{Sigma_tilde}
 {\widetilde \Sigma}_{11} := & \widehat \Sigma_{11} - \widehat
\Sigma_{12} \widehat \Sigma_{22}^{-1} \widehat \Sigma_{21}.
\end{align}

\begin{proposition}\label{prop:hd1}
If $(H,D)$ is a stationary point of the algorithm with $D_2=0$, then
\[
\ii(\widehat{\Sigma}||HH^{\top}+D)=\ii(\widetilde{\Sigma}_{11}||\widetilde{H}_{1}\widetilde{H}_{1}^{\top}+D_{1}).
\]
Moreover, the stationary equations~(\ref{stateqs}) reduce to
\begin{align*}
\widetilde{H}_{1} & =\widetilde{\Sigma}_{11}(\widetilde{H}_{1}\widetilde{H}_{1}^{\top}+D_{1})^{-1}\widetilde{H}_{1}=
(\widetilde\Sigma_{11}-\widetilde{H}_{1}\widetilde{H}_{1}^{\top})D_1^{-1}\widetilde{H}_1 \\
D_1 & =
\Delta(\widetilde\Sigma_{11}-\widetilde{H}_{1}\widetilde{H}_{1}^{\top}).
\end{align*}
\end{proposition}
\begin{proof}
One easily verifies that for any nonsingular matrix $A$ of the
appropriate size
$$\ii(APA^{\top }||AQA^{\top}) =\ii(P||Q).$$
Taking
\[
A=
\begin{pmatrix}
I & -\widehat{\Sigma}_{12}\widehat{\Sigma}_{22}^{-1} \\ 0 & I
\end{pmatrix},
\]
one finds
\[
A\widehat{\Sigma}A^{\top}=
\begin{pmatrix}
\widetilde{\Sigma}_{11} & 0 \\ 0 & \widehat{\Sigma}_{22}
\end{pmatrix}.
\]
Moreover, by Proposition~\ref{prop:d2},
\[
A(HH^{\top}+D)A^{\top}=
\begin{pmatrix}
H_{1}H_{1}+D_{1}-\widehat{\Sigma}_{12}\widehat{\Sigma}_{22}^{-1}\widehat{\Sigma}_{21} & 0 \\
0 & \widehat{\Sigma}_{22},
\end{pmatrix}.
\]
where the upper left block is equal to
$\widetilde{H}_{1}\widetilde{H}_{1}^{\top}+D_{1}$ in view of
equation~(\ref{HtildeHtildeS}). The first assertion follows. The
reduced stationary equations follow by simple computation.
\end{proof}
\begin{remark} Under the conditions of Proposition~\ref{prop:hd1},
the pair $(\widetilde{H}_{1},D_{1})$ is also a stationary point for
the minimization of
$\ii(\widetilde{\Sigma}_{11}||\widetilde{H}_{1}\widetilde{H}_{1}^{\top}+D_{1})$.
This is in full agreement with the results of
Section~\ref{sec:singularD}.
\end{remark}
%


\subsection{Approximation with singular $D$}\label{sec:singularD}

In this section we consider the approximation Problem~\ref{problem2}
under the constraint $D_2=0$. \cite{joreskog}
investigated the solution of the likelihood equations~(\ref{eq:mld})
and~(\ref{eq:mlh2}) under zero constraints on $D$, whereas in this
section we work directly on the objective function of
Problem~\ref{problem2} without referring to those equations.
The constrained minimization problem can be formulated as
\begin{problem}\label{problem:dis0}
Given $\widehat{\Sigma}>0$ of size $n\times n$ and integers $n_2$
and $k$, with $n_2 \le k < n$, minimize
\begin{equation}
\ii(\widehat{\Sigma} || HH^\top+D),
\end{equation}
over $(H, D)$ with $D$ satisfying~(\ref{eq:dd1}).
\end{problem}
We will now decompose the objective function, choosing a convenient
representation of the matrix $H$, in order to reduce the complexity
of Problem~\ref{problem:dis0}. To that end we make the following
observation. Given any orthogonal matrix $Q$, define $H' = HQ$, then
clearly $H'H'^\top + D = HH^\top +D$. Let $H_{2}=U(0\,\,
\Lambda)V^{\top}$ be the singular value decomposition of $H_{2}$,
with $\Lambda$ a positive definite diagonal matrix of size
$n_{2}\times n_{2}$, and $U$ and $V$ orthogonal of sizes
$n_{2}\times n_{2}$ and $k\times k$ respectively. Let
$$H'= HV$$
The blocks of $H'$ are $H_{1}'  =H_{1}V$ and $H_{2}'=(H_{21}'\,
H_{22}'):=(0\quad U\Lambda)$, with $H_{21}'\in\R^{(k-n_2)\times
n_2}$ and $H'_{22}\in\R^{n_2\times n_2}$. Hence, without loss of
generality, in the remainder of this section we assume that
\begin{equation}
H=\begin{pmatrix} H_1 \\ H_2
\end{pmatrix} = \begin{pmatrix} H_{11} & H_{12} \\
0 & H_{22}
\end{pmatrix}\label{eq:h21}, \qquad H_{22} \,\, \mbox{invertible}.
\end{equation}
Finally, let
$$K=\widehat{\Sigma}_{12}\widehat{\Sigma}_{22}^{-1}-H_1H_2^\top(H_2H_2^\top)^{-1},$$
which, under~(\ref{eq:h21}), is equivalent to
$$K=\widehat{\Sigma}_{12}\widehat{\Sigma}_{22}^{-1}-H_{12}H_{22}^{-1}.$$
Here is the announced decomposition of the objective function.
\begin{lemma}\label{lemma:divdecomposition} Let $D$ be as in equation~(\ref{eq:dd1}).
The  following I-divergence decomposition holds.
\begin{align}\label{eq:divdivtr}
\ii(\widehat{\Sigma}||HH^{\top}+D) & = \ii(\widetilde{\Sigma}_{11}||H_{11}H_{11}^{\top}+
D_{1})+\ii(\widehat{\Sigma}_{22}||H_{22}H_{22}^{\top})\nonumber \\
& \quad \mbox{}+\half \mbox{\rm
tr}\big(\widehat{\Sigma}_{22}K^\top(H_{11}H_{11}^{\top}+D_{1})^{-1}K\big).
\end{align}
\end{lemma}
\begin{proof}
See Appendix~\ref{section:tech}.
\end{proof}
We are now ready to characterize the solution of
Problem~\ref{problem:dis0}.
\begin{proposition}\label{prop:mind2}
Any pair $(H, D)$, as in~(\ref{eq:dd1}) and~(\ref{eq:h21}), solving
Problem~\ref{problem:dis0} satisfies
$$
\begin{array}{ll} \bullet  &
\ii(\widetilde{\Sigma}_{11}||H_{11}H_{11}^{\top}+D_{1}) \quad
\mbox{is minimized}, \\
\bullet  &  H_{22}H_{22}^{\top}  = \widehat{\Sigma}_{22}, \\
\bullet  &  H_{12} =
\widehat{\Sigma}_{12}\widehat{\Sigma}_{22}^{-1}H_{22}.
\end{array}
$$
\end{proposition}
\begin{proof}
Observe first that the second and third terms on the right hand side
of~(\ref{eq:divdivtr}) are nonnegative and can be made zero. To this
end it is enough to select $H_{22}$ such that
$H_{22}H_{22}^{\top}=\widehat{\Sigma}_{22}$ and then
$H_{12}=\widehat{\Sigma}_{12}\widehat{\Sigma}_{22}^{-1}H_{22}$. The
remaining blocks, $H_{11}$ and $D_1$, are determined minimizing the
first term.
\end{proof}
\begin{remark}\label{remark:expl}
In the special case $n_{2}=k$, the matrices $H_{11}$ and $H_{21}$
are empty, $H_{12}=H_{1}$, and $H_{22}=H_{2}$. From
Proposition~\ref{prop:mind2}, at the minimum,
$H_{2}H_{2}^{\top}=\widehat{\Sigma}_{22}$,
$H_{1}H_2^\top=\widehat{\Sigma}_{12}$, and $D_{1}$ minimizes
$\ii(\widetilde{\Sigma}_{11}||D_{1})$. The latter problem has
solution $D_{1}=\Delta(\widetilde{\Sigma}_{11})$. It is remarkable
that in this case the minimization problem has an {\em explicit}
solution.
\end{remark}
Proposition~\ref{prop:mind2} also sheds some light on the
unconstrained Problem~\ref{problem2}.


\begin{corollary}\label{cor:sigmad0}
Assume that, in Problem~\ref{problem2},
$\ii(\widehat{\Sigma}||HH^{\top}+D)$ is minimized for a pair $(H,D)$
with $D$ of the form~(\ref{eq:dd1}). Then
$\widehat{\Sigma}_{12}=H_{1}H_{2}^{\top}$,
$\widehat{\Sigma}_{22}=H_{2}H_{2}^{\top}$, and
$(\widetilde{H}_{1},D_{1})$, where $\widetilde{H}_1$ is as
in~(\ref{Htilde}), minimizes
$\ii(\widetilde{\Sigma}_{11}||\widetilde{H}_{1}\widetilde{H}_{1}^{\top}+D_{1})$.
\end{corollary}

\begin{proof}
It is obvious that, in this case, Problem~\ref{problem2} and
Problem~\ref{problem:dis0} are equivalent. The result readily
follows from Proposition~\ref{prop:mind2}, in view of the equality
$\widetilde{H}_{1}=(H_{11}\,0)$.
\end{proof}
Hence, in Problem~\ref{problem2}, a singular minimizer $D$ can occur
only if $\widehat{\Sigma}$ has the special structure described in
Corollary~\ref{cor:sigmad0}.

%

In the same spirit one can characterize the covariances $\widehat \Sigma$,
admitting an exact FA model of size $k \ge n_2$, and with
$D_2=0$. This happens if and only if, with the notations of Section~\ref{section:statpoints},
\begin{equation} \label{exactFA}
\widehat \Sigma_{11} - \widehat \Sigma_{12} \widehat\Sigma_{22}^{-1} \widehat\Sigma_{21} \quad \mbox{is a diagonal matrix.}
\end{equation}
This condition is easily interpreted in terms of random vectors. Let
$Y$ be an $n$ dimensional, zero mean, normal vector with
$\cov(Y)=\widehat \Sigma$ and partition it into two subvectors
$(Y_1, Y_2)$, of respective sizes $n_1$ and $n_2$, corresponding to
the block partitioning of $\widehat \Sigma$. The above condition
states that the components of $Y_1$ are conditionally independent
given $Y_2$. The construction of the $k$-dimensional, exact FA
model, with $D_2=0$ is as follows.


Let $D_1=\Sigma_{11}-\Sigma_{12}\Sigma_{22}^{-1}\Sigma_{21}$, which
is a diagonal matrix by assumption. Let $R$ be the symmetric,
invertible, square root of $\Sigma_{22}$. Define the matrices
\begin{align*}
H_1 & = \Sigma_{12}\begin{pmatrix} R^{-1} & 0\end{pmatrix}\in \openR^{n_1\times k} \\
H_2 & =\begin{pmatrix} R & 0
\end{pmatrix}\in\openR^{n_2\times k}.
\end{align*}
One verifies the identities $H_2H_2^\top=\Sigma_{22}$,
$H_1H_2^\top=\Sigma_{12}$ and
$H_1H_1^\top=\Sigma_{12}\Sigma_{22}^{-1}\Sigma_{21}$. It follows
that $\Sigma_{11}=D_1+H_1H_1^\top$. Let $Z$ be a
$(k-n_2)$-dimensional random vector, independent of $Y$, with zero
mean and identity covariance matrix. Put
\[
X=\begin{pmatrix} R^{-1}Y_2 \\Z
\end{pmatrix}.
\]
Then $\cov(X)=I_k$. Furthermore, $\eps_1:=Y_1-H_1X$ is independent
of $X$ with $\cov(\eps_1)=D_1$, and $Y_2-H_2X=0$. It follows that
\begin{align*}
Y_1 & = H_1 X+\eps_1 \\
Y_2 & = H_2 X
\end{align*}
is an exact realization of $Y$ in terms of a factor model.

\subsection{Algorithm when a part of $D$ has zero diagonal}
\label{sec:algo0}

In Section~\ref{sec:singularD} we have posed the minimization
problem under the additional constraint that the matrix $D$ contains
a number of zeros on the diagonal. In the present section we
investigate how this constraint affects the alternating minimization
algorithm. For simplicity we give a detailed account of this, only
using the recursion~(\ref{eq:hh2}) for $\ch_t$. Initialize the
algorithm at $(H_0, D_0)$ with
\begin{equation}\label{eq:d10} D_0 = \begin{pmatrix} \widetilde{D} &
0 \\ 0 & 0
\end{pmatrix},
\end{equation}
where $\widetilde{D}>0$ is of size $n_1\times n_1$ and
\begin{equation}\label{eq:h12}
H_0=\begin{pmatrix} H_1 \\H_2
\end{pmatrix},
\end{equation}
where $H_2\in\R^{n_2\times k}$ is assumed to have full row rank, so
that $n_2\leq k$ (note the slight ambiguity in the notation for the
blocks of $H_0$). Clearly $H_0H_0^\top + D_0$ is invertible. For
$H_0$ as in equation~(\ref{eq:h12}) put
\begin{equation}\label{eq:hhtilde}
\widetilde{\ch}= H_1(I-H_2^\top(H_2H_2^\top)^{-1}H_2)H_1^\top.
\end{equation}
We have the following result.

\begin{lemma}\label{prop:hh11}
Let $(H_0, D_0)$ be given as above, and $\ch_0=H_0H_0^\top$.
Applying one step of recursion~(\ref{eq:hh2}), one gets
\begin{align}
\ch_1 & = \begin{pmatrix} \ch^{11} & \widehat{\Sigma}_{12} \\
\widehat{\Sigma}_{21} & \widehat{\Sigma}_{22}
\end{pmatrix},\label{eq:hh11}
\end{align}
where
\begin{equation}
\ch^{11} =
\widetilde{\Sigma}_{11}(\widetilde{\ch}+
\widetilde{D})^{-1}\widetilde{\ch}(\widetilde{D}+\widetilde{\Sigma}_{11}
(\widetilde{\ch}+\widetilde{D})^{-1}\widetilde{\ch})^{-1}\widetilde{\Sigma}_{11}
+\widehat{\Sigma}_{12}\widehat{\Sigma}_{22}^{-1}\widehat{\Sigma}_{21}.
\end{equation}
and
\begin{align}
D_1 & = \begin{pmatrix} \Delta(\widetilde{\Sigma}_{11}-\widetilde{\ch}) & 0 \\
0 & 0
\end{pmatrix}.
\end{align}
\end{lemma}

\begin{proof}
We start from Equation~(\ref{eq:hh2}) with $t=0$ and compute the
value of $\ch_1$. To that end we first obtain under the present
assumption an expression for the matrix $(\ch+D_0)^{-1}\ch$. Let
$P=I-H_2^\top(H_2H_2^\top)^{-1}H_2$. It holds that
\begin{equation}\label{eq:hdh}
(\ch+D_0)^{-1}\ch=
\begin{pmatrix}
(\widetilde{D} +H_1PH_1^\top)^{-1}H_1PH_1^\top & 0 \\
(H_2H_2^\top)^{-1}H_2H_1^\top(\widetilde{D}
+H_1PH_1^\top)^{-1}\widetilde{D} & I
\end{pmatrix},
\end{equation}
as one can easily verify by multiplying this equation by $\ch+D_0$.
We also need the inverse of $D_0+\widehat{\Sigma}(\ch+D_0)^{-1}\ch$,
postmultiplied with $\widehat{\Sigma}$. Introduce
$U=\widetilde{D}+\widetilde{\Sigma}_{11}(H_1PH_1^\top+\widetilde{D})^{-1}H_1PH_1^\top$
and
\[
V=\widehat{\Sigma}_{22}^{-1}\widehat{\Sigma}_{21}(H_1PH_1^\top +
\widetilde{D})^{-1}+(H_2H_2^\top)^{-1}H_2H_1^\top(H_2H_2^\top)^{-1}\widetilde{D}.
\]
It results that
\begin{equation}\label{eq:uv}
\big(D_0+\widehat{\Sigma}(\ch+D_0)^{-1}\ch\big)^{-1}\widehat{\Sigma}=
\begin{pmatrix}
U^{-1}\widetilde{\Sigma}_{11} & 0 \\
-VU^{-1}\widetilde{\Sigma}_{11}+\widehat{\Sigma}_{22}^{-1}\widehat{\Sigma}_{21}
& I
\end{pmatrix}.
\end{equation}
Insertion of the expressions~(\ref{eq:hdh}) and~(\ref{eq:uv})
into~(\ref{eq:hh2}) yields the result.
\end{proof}
The update equations of the algorithm for $\ch_t$ and $D_t$, can be
readily derived from Lemma~\ref{prop:hh11} and are summarized below.
\begin{proposition} \label{algo:singD}
The upper left block $\ch^{11}_t$ of $\ch_t$, can be computed
running a recursion for $\widetilde{\ch}_t :=
\ch^{11}_t-\widehat{\Sigma}_{12}\widehat{\Sigma}_{22}^{-1}\widehat{\Sigma}_{21}$,
\[
\widetilde{\ch}_{t+1}=
\widetilde{\Sigma}_{11}(\widetilde{\ch}_t+\widetilde{D}_t)^{-1}\widetilde{\ch}_t(\widetilde{D}_t+
\widetilde{\Sigma}_{11}(\widetilde{\ch}_t+\widetilde{D}_t)^{-1}\widetilde{\ch}_t)^{-1}\widetilde{\Sigma}_{11},
\]
whereas the blocks on the border of $\ch_t$ remain constant. The
iterates for $D_t$ all have a lower right block of zeros, while the
upper left $n_1 \times n_1$ block $\widetilde{D}_t$ satisfies
$$
\widetilde{D}_{t} = \Delta(\widetilde{\Sigma}_{11} -
\widetilde{\ch}_{t}).
$$
\mbox{}\hfill $\square$
\end{proposition}

%
\noindent Note that the recursions of Proposition~\ref{algo:singD}
are exactly those that follow from the optimization
Problem~\ref{problem:dis0}. Comparison with~(\ref{eq:hh2}), shows
that, while the algorithm for the unconstrained case updates $\ch_t$
of size $n\times n$, now one needs to update $\widetilde{\ch}_t$
which is of smaller size $n_1\times n_1$.
\medskip\\
Now we specialize the above to the case in which $n_2=k$.

\begin{corollary}
Let the initial value $D_0$ be as in Equation~(\ref{eq:d10}) with
$n_2=k$. Then for any initial value $\ch_0$ the algorithm converges
{\em in one step} and one has that the first iterates $D_1$ and
$\ch_1$,  which are equal to the terminal values, are given by
\begin{align*}
D_1 & = \begin{pmatrix} \Delta(\widetilde{\Sigma}_{11}) & 0 \\ 0 & 0
\end{pmatrix} \\
\ch_1 & = \begin{pmatrix}
\widehat{\Sigma}_{12}\widehat{\Sigma}_{22}^{-1}\widehat{\Sigma}_{21}
& \widehat{\Sigma}_{12} \\ \widehat{\Sigma}_{21} &
\widehat{\Sigma}_{22}
\end{pmatrix}.
\end{align*}
\end{corollary}

\begin{proof}
We use Proposition~\ref{prop:hh11} and notice that in the present
case the matrix $\widetilde{\ch}$ of (\ref{eq:hhtilde}) is equal to
zero. Therefore
$\widetilde{\ch}^{11}=\widehat{\Sigma}_{12}\widehat{\Sigma}_{22}^{-1}\widehat{\Sigma}_{21}$
and the result follows.
\end{proof}
It is remarkable that in this case we have convergence of the
iterates in one step only. Moreover the resulting values are exactly
the theoretical ones, which we have explicitly computed in
Remark~\ref{remark:expl}.

\section{Numerical examples}\label{section:numerics}

\subsection{Simulated data} In the present section we investigate the performance of the
Algorithm~\ref{algo1} and compare it to the behaviour of the EM
Algorithm~\ref{em}. In all examples we take $\widehat{\Sigma}$ equal
to $AA^\top+c\,{\rm diag}(d)$, where $A\in\R^{n\times m}$ with
$m\leq n$, $d\in\R^{n}_+$ and $c\geq 0$, for various values of
$n,m$. The matrices $A$ and the vector $d$ have been randomly
generated. The notation $A$=rand$(n,m)$ means that $A$ is a randomly
generated matrix of size $n\times m$, whose elements are
independently drawn from the uniform distribution on $[0,1]$. In all
cases the resulting matrix $\widehat{\Sigma}$ is strictly positive
definite. The reason for incorporating the component $d$ is that we
want to check whether the algorithm is able to reconstruct
$\widehat{\Sigma}=AA^\top+{\rm diag}(d)$ in case the inner size $k$
of the matrices $H_t$ produced by the algorithm is taken to be equal
to $m$.

We have also included results on the $L^2$-norm of the difference
between the given matrix $\widehat{\Sigma}$ and its approximants
$\Sigma_t=H_tH_t^\top+D_t$, i.e. we also compute
$\ell_t=\big(\tr((\widehat{\Sigma}-\Sigma_t)^\top(\widehat{\Sigma}-\Sigma_t))\big)^{1/2}$.
The origin of this extra means of comparison of behavior of the
Algorithms~\ref{algo1} and~\ref{em} is that we detected in a number
of cases that in terms of the value of the divergences, the
difference between the approximations generated by the two
algorithms was, after enough iterations, negligible, whereas a
simple look at the matrices produced by the final iterations
revealed that Algorithm~\ref{algo1} produced very acceptable, if not
outstanding results, whereas the approximations generated by the EM
algorithm~\ref{em} for the same given matrix were rather poor. This
phenomenon is reflected by a huge $L^2$-error of the EM algorithm,
as compared to a small one of Algorithm~\ref{algo1}. The choice for
the $L^2$-norm is to some extent arbitrary. We are basically
concerned with good approximations in terms of I-divergence, and it
is therefore a priori not completely fair to judge the quality of
approximations by switching to another criterion. However, the
$L^2$-norm of the error has an intuitive interpretation, is easy to
compute and also has some appealing properties in the context of the
two partial minimization problems, cf.\ Remarks~\ref{remark:inv}
and~\ref{remark:l2d}.

We have plotted various characteristics of the algorithms against
the number of iterations, for both of them the divergence at each
iteration, as well as their counterparts for the $L^2$-norm
(dashed lines). For reasons of clarity, in all figures we have
displayed the characteristics on a logarithmic scale.
\medskip
\begin{center}\begin{tabular}{ll}
Legenda &  \\  \hline \\ [-1.8ex]
solid blue: & divergence $\ii(\widehat\Sigma||\widehat\Sigma_t)$ in algorithm~\ref{algo1} \\
solid red: & divergence in the EM algorithm~\ref{em} \\
dashed blue: & $L^2$-norm of $\widehat\Sigma - \widehat\Sigma_t$ in algorithm~\ref{algo1} \\
dashed red: & $L^2$-norm in the EM algorithm~\ref{em} \\
\end{tabular}
\end{center}\medskip
The numerical results have been obtained by running Matlab.

\begin{figure}
\begin{center}
\includegraphics*[viewport=80 220 540 560,scale=0.7]{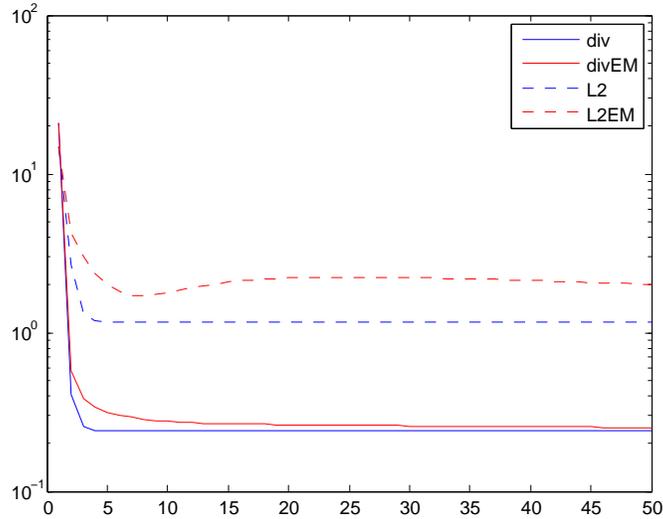}
\vspace{-5ex}\caption{{\rm $A$=rand(10,5) , $d=2*$rand(10,1) , $k$=2}}
\end{center}
\end{figure}

\begin{figure}
\begin{center}
\includegraphics*[viewport=80 220 540 560,scale=0.7]{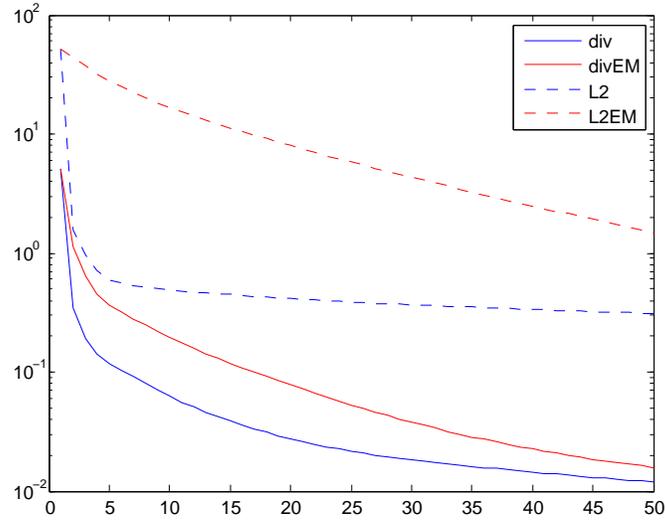}
\vspace{-5ex}\caption{{\rm $A$=rand(10,5) , $d=2*$rand(10,1) , $k$=5}}
\end{center}
\end{figure}

\begin{figure}
\begin{center}
\includegraphics*[viewport=80 220 540 560,scale=0.7]{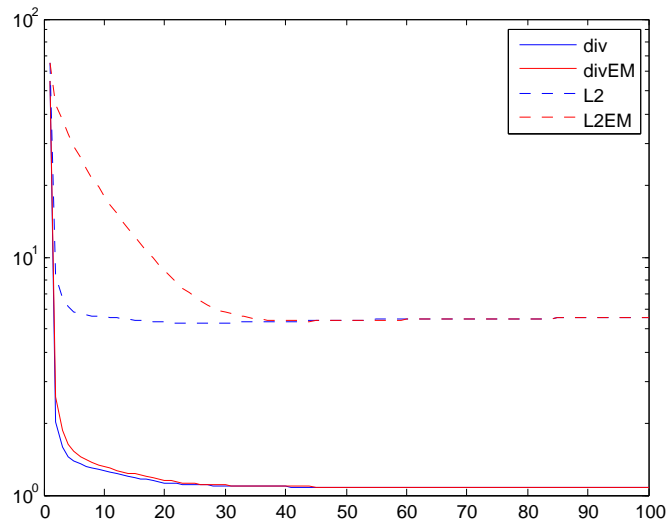}
\vspace{-5ex}\caption{{\rm $A$=rand(30,15) , $d=3*$rand(30,1) , $k$=5}}
\end{center}
\end{figure}

\begin{figure}
\begin{center}
\includegraphics*[viewport=80 220 540 560,scale=0.7]{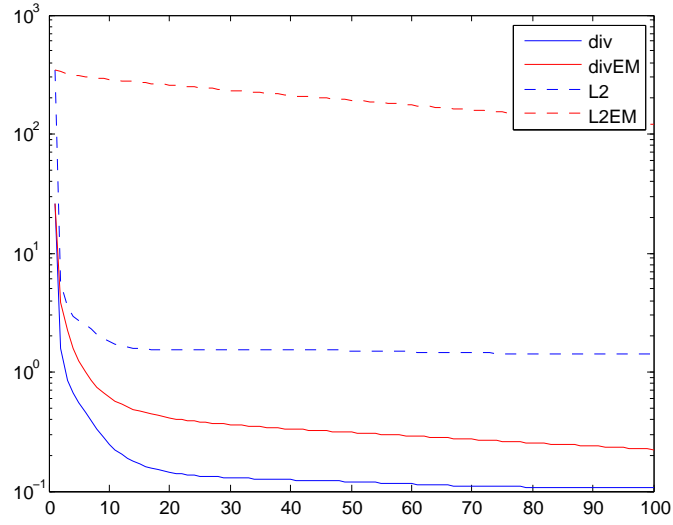}
\vspace{-5ex}\caption{{\rm $A$=rand(30,15) , $d=3*$rand(30,1) , $k$=15}}
\end{center}
\end{figure}

\begin{figure}
\begin{center}
\includegraphics*[viewport=80 220 540 560,scale=0.7]{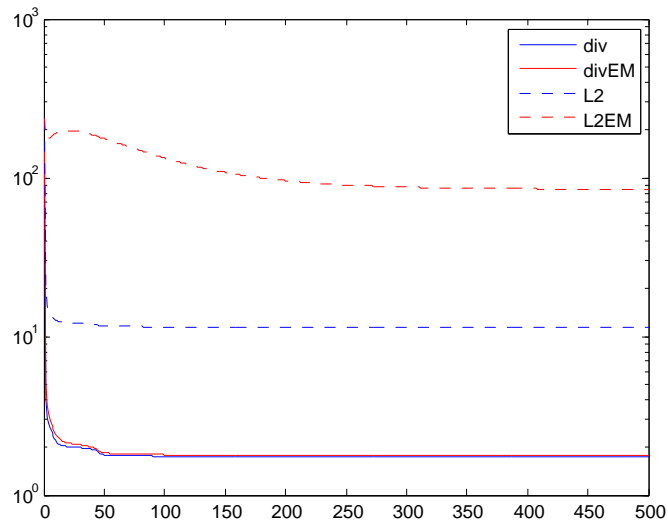}
\vspace{-5ex}\caption{{\rm $A$=rand(50,30) , $d=5*$rand(50,1) , $k$=10}}
\end{center}
\end{figure}

\begin{figure}
\begin{center}
\includegraphics*[viewport=80 220 540 560,scale=0.7]{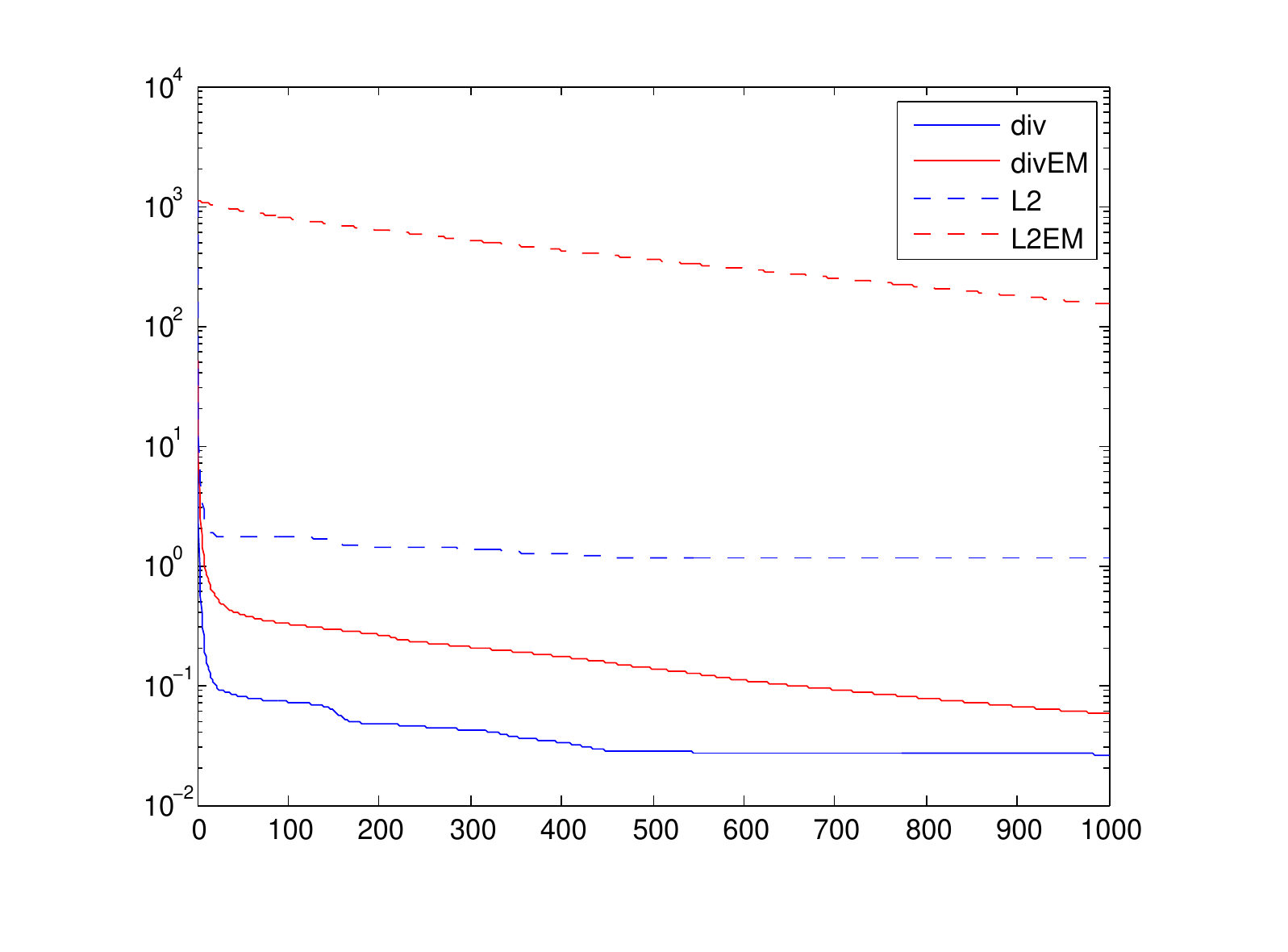}
\vspace{-5ex}\caption{{\rm $A$=rand(50,30) , $d=5*$rand(50,1) , $k$=30}}
\end{center}
\end{figure}

Figures 1 and 2 show the behaviour of the two algorithms in cases
with $n=10$ (which is relatively small) and for $k=2,5$
respectively. We observe that the performance of the algorithms
hardly differ, especially for $k=2$. In Figures 3 and 4, we have
$n=30$ and $k=5,15$ respectively. We notice that in terms of
divergence, the performance of the two algorithms is roughly the
same for $k=5$, but for $k=15$ there are noticeable differences.
But looking at the $L^2$-norm of the error, we even see a manifest
difference of the outcomes. The differences are even more
pronounced in Figures 5 and 6, where $n=50$ and $k=10, 30$
respectively. In the former case, in terms of divergences, the two
algorithms behave roughly the same, but there is a factor 10 of
difference in the $L^2$-erros. In the latter case, where $k=m$ one
would expect that both algorithms are able to retrieve the
original matrix $\widehat{\Sigma}$, which seems to be the case,
although Algorithm~\ref{algo1} behaves the best. Looking at the
$L^2$-error, we see a gross difference between the
Algorithm~\ref{algo1} and the EM algorithm of order about 100.
This striking difference in behaviour between the two algorithms
is typical.

\subsection{Real data example}

In the present section we test our algorithm on the data provided in
the original paper \cite{rubinthayer1982}, where the EM algorithm
for FA models has been presented first. The results, with in this
case $\widehat{\Sigma}$ the empirical correlation matrix of the
data, are presented in Figure~7. We observe that again
Algorithm~\ref{algo1} outperforms the EM algorithm. The underlying
numerical results are at first sight very close to those
of~\cite{rubinthayer1982} (we have also taken $k=4$), but we observe
like in the previous section that the convergence of the EM
algorithm is much slower than that of Algorithm~\ref{algo1} and
after 50 iterations (the same number as in~\cite{rubinthayer1982})
the differences are quite substantial.

\begin{figure}
\begin{center}
\includegraphics*[viewport=80 220 540 560,scale=0.7]{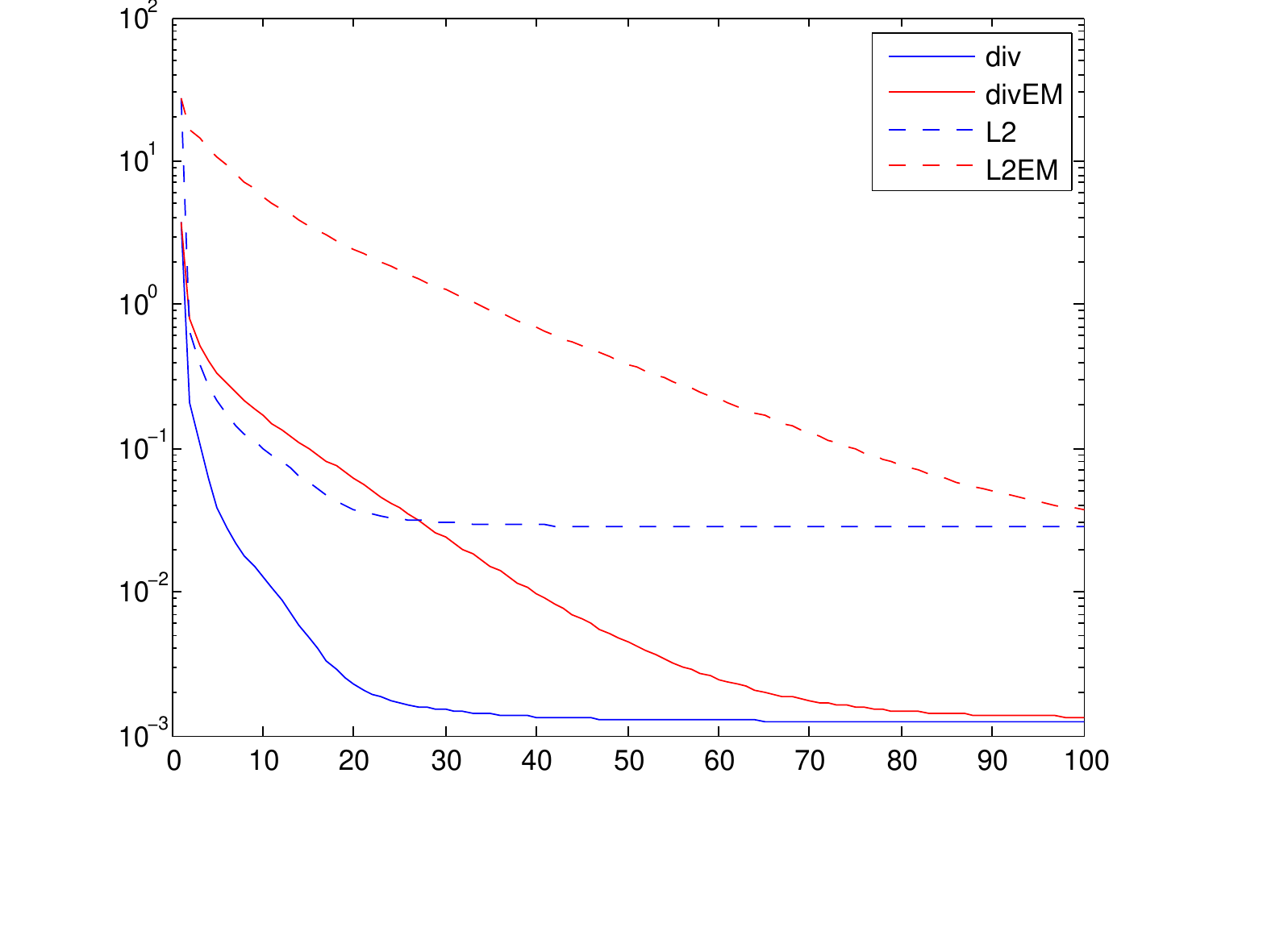}
\vspace{-10ex}\caption{Rubin-Thayer data, $k$=4}
\end{center}
\end{figure}


\appendix

%

\section{Multivariate normal
distribution}\label{gauss}
\setcounter{equation}{0}

Let $(X^\top,Y^\top)^\top$ be a zero mean normally distributed
random vector with covariance matrix
\[
\Sigma=\begin{pmatrix} \Sigma_{XX} & \Sigma_{XY} \\
\Sigma_{YX} & \Sigma_{YY}
\end{pmatrix}.
\]
Assume that $\Sigma_{YY}$ is invertible, then the conditional
distribution of $X$ given $Y$ is normal, with mean vector $\ee[X|Y]
=\Sigma_{XY}\Sigma_{YY}^{-1}Y$ and covariance matrix
\begin{equation}\label{eq:condcov}
\cov[X|Y]=\Sigma_{XX}- \Sigma_{XY}\Sigma_{YY}^{-1}\Sigma_{YX}.
\end{equation}
Consider two normal distributions $\nu_1=N(\mu_1,\Sigma_1)$ and
$\nu_2=N(\mu_2,\Sigma_2)$ on a common Euclidean space. The
I-divergence is easily computed as
\begin{align}\label{eq:divmu}
\ii(\nu_1||\nu_2) & =\half \log
\frac{|\Sigma_2|}{|\Sigma_1|}-\frac{m}{2} +\half {\rm
tr}(\Sigma_2^{-1}\Sigma_1)+\half(\mu_1-\mu_2)^\top\Sigma_2^{-1}
(\mu_1-\mu_2) \nonumber \\
& = \ii(\Sigma_1||\Sigma_2) + \half(\mu_1-\mu_2)^\top\Sigma_2^{-1}
(\mu_1-\mu_2),
\end{align}
where $\ii(\Sigma_1||\Sigma_2)$ denotes as before the I-divergence
between positive definite matrices. The extra term, depending on the
nonzero means, did not appear in~(\ref{eq:divsigma}).

\section{Matrix identities}
\setcounter{equation}{0}

For ease of reference we collect here some well known identities from matrix algebra.

\medskip
The following lemma is verified by a straightforward computation.
\begin{lemma}\label{lemma:matrix}
Let $A,B,C,D$ be blocks of compatible sizes of a given matrix, with
$A$ and $D$ both square. If $D$ is invertible the following
decomposition holds
\[
\begin{pmatrix}
A & C \\B & D
\end{pmatrix}=
\begin{pmatrix}
I & CD^{-1} \\ 0 & I
\end{pmatrix}
\begin{pmatrix}
A-CD^{-1}B & 0 \\ 0 & D
\end{pmatrix}
\begin{pmatrix}
I & 0 \\D^{-1}B & I
\end{pmatrix}
\]
while, if $A$ is invertible, the following decomposition holds
\[
\begin{pmatrix}
A & C \\B & D
\end{pmatrix}=
\begin{pmatrix}
I & 0 \\ BA^{-1} & I
\end{pmatrix}
\begin{pmatrix}
A & 0 \\ 0 & D-BA^{-1}C
\end{pmatrix}
\begin{pmatrix}
I & A^{-1}C \\0 & I
\end{pmatrix}.
\]
Furthermore, assuming that $A$, $D$, and $A-CD^{-1}B$ are all
invertible, we have
\begin{align*}
& \begin{pmatrix} A & C \\B & D
\end{pmatrix}^{-1}= \\
& \quad \begin{pmatrix} (A-CD^{-1}B)^{-1} &
-(A-CD^{-1}B)^{-1}CD^{-1}
\\-D^{-1}B(A-CD^{-1}B)^{-1} &
D^{-1} + D^{-1}B(A-CD^{-1}B)^{-1}CD^{-1}
\end{pmatrix}.
\end{align*}
\end{lemma}
\begin{corollary}\label{cor:inv}
Let $A,B,C,D$ matrices as in Lemma~\ref{lemma:matrix} with $A$, $D$,
and $A-CD^{-1}B$ all invertible. Then $D-BAC$ is also invertible,
with
\[
(D-BAC)^{-1}=D^{-1}+D^{-1}B(A^{-1}-CD^{-1}B)^{-1}CD^{-1}.
\]
\end{corollary}
\begin{proof}
Use the two decompositions of Lemma~\ref{lemma:matrix} with $A$
replaced by $A^{-1}$ and compute the two expressions of the lower
right block of the inverse matrix.
\end{proof}

\begin{corollary}\label{cor:ibc}
Let $B\in\R^{n\times m}$ and $C\in\R^{m\times n}$. Then $\det
(I_n-BC)=\det(I_m-CB)$ and $I_n-BC$ is invertible if and only if
$I_m-CB$ is invertible.
\end{corollary}
\begin{proof}
Use the two decompositions of Lemma~\ref{lemma:matrix} with
$A=I_m$ and $D=I_n$ to compute the determinant of the block
matrix.
\end{proof}

\begin{corollary}\label{cor:ihht}
Let $D$ be a positive definite matrix, not necessarily strictly
positive definite. If $HH^\top+D$ is strictly positive definite then also
$I-H^\top(HH^\top + D)^{-1}H$ is strictly positive.
\end{corollary}
\begin{proof}
Use Lemma~\ref{lemma:matrix} with $A=I$, $B=H$, $C=H^\top$ and $D$
replaced with  $HH^\top + D$. The two middle matrices in the
decompositions are respectively
\[
\begin{pmatrix}
I-H^\top(HH^\top + D)^{-1}H & 0 \\
0 & HH^\top + D
\end{pmatrix}
\]
and
\[
\begin{pmatrix}
I & 0 \\
0 & D
\end{pmatrix}.
\]
Hence, from the second decomposition it follows from positive
definiteness of $D$ that $\begin{pmatrix} I & H^\top \\ H &
HH^\top + D
\end{pmatrix}$ is positive definite, and then from the first
decomposition that $I-H^\top(HH^\top + D)^{-1}H$ is positive
definite.
\end{proof}

\section{Decompositions of the I-divergence}\label{section:decdiv}
\setcounter{equation}{0}

We derive here a number of decomposition results for the
I-divergence between two probability measures. Similar results are
derived in~\cite{cramer2}, see also~\cite{finessospreij} for the
discrete case. These decompositions yield the core arguments for the
proofs of the propositions in Sections~\ref{section:pms}
and~\ref{sec:singularD}.
\begin{lemma}\label{lemma:xcondy}
Let $\pp_{XY}$ and $\qq_{XY}$ be given probability distributions of
a Euclidean random vector $(X,Y)$ and denote by $\pp_{X|Y}$ and
$\qq_{X|Y}$ the corresponding regular conditional distributions of
$X$ given $Y$. Assume that $\pp_{XY} \ll \qq_{XY}$. Then
\begin{equation}\label{eq:divsplit}
\ii(\pp_{XY}||\qq_{XY})=\ii(\pp_Y||\qq_Y)+\mathbb{E}_{\pp_Y}\ii(\pp_{X|Y}||\qq_{X|Y}).
\end{equation}
\end{lemma}
\begin{proof} It is easy to see that we also have $\pp_Y \ll \qq_Y$.
Moreover we also have absolute continuity of the conditional laws,
in the sense that if $0$ is a version of the conditional probability
$\qq(X\in B|Y)$, then it is also a version of $\pp(X\in B|Y)$. One
can show that a conditional version of the Radon-Nikodym theorem
applies and that a conditional Radon-Nikodym derivative
$\frac{\dd\pp_{X|Y}}{\dd\qq_{X|Y}}$ exists $\qq_Y$-almost surely.
Moreover, one has the $\qq_{XY}$-a.s.\ factorization
\[
\frac{\dd\pp_{XY}}{\dd\qq_{XY}}=\frac{\dd\pp_{X|Y}}{\dd\qq_{X|Y}}\frac{\dd\pp_{Y}}{\dd\qq_{Y}}.
\]
Taking logarithms on both sides and expectation under $\pp_{XY}$
yields
\[
\mathbb{E}_{\pp_{XY}}\log\frac{\dd\pp_{XY}}{\dd\qq_{XY}}=\mathbb{E}_{\pp_{XY}}\log
\frac{\dd\pp_{X|Y}}{\dd\qq_{X|Y}}+\mathbb{E}_{\pp_{XY}}\log\frac{\dd\pp_{Y}}{\dd\qq_{Y}}.
\]
Writing the first term on the right hand side as
$\mathbb{E}_{\pp_{XY}}\{\mathbb{E}_{\pp_{XY}}[\log\frac{\dd\pp_{X|Y}}{\dd\qq_{X|Y}}|Y]\}$,
we obtain
$\mathbb{E}_{\pp_{Y}}\{\mathbb{E}_{\pp_{X|Y}}[\log\frac{\dd\pp_{X|Y}}{\dd\qq_{X|Y}}|Y]\}$.
The result follows.
\end{proof}
The decomposition of Lemma~\ref{lemma:xcondy} is useful when solving
I-divergence minimization problems with marginal constraints, like
the one considered below.
\begin{proposition}\label{prop:gpm1}
Let $\qq_{XY}$ and $\pp^0_Y$ be given probability distributions of a
Euclidean random vector $(X,Y)$, and of its subvector $Y$
respectively. Consider the I-divergence minimization problem
$$
\min_{\pp_{XY} \in \mathcal P} \, \ii(\pp_{XY}||\qq_{XY}),$$ where
$$
\mathcal P := \{ \pp_{XY} \, | \, \int \pp_{XY}(dx, Y) =  \pp^0_Y
\}.
$$
If the marginal $\pp^0_Y \ll \qq^0_Y$, then the I-divergence is
minimized by $\pp^*_{XY}$ specified by the Radon-Nikodym derivative
\begin{equation}\label{eq:rn}
\frac{\dd\pp^*_{XY}}{\dd\qq_{XY}}=\frac{\dd\pp^0_Y}{\dd\qq_Y}.
\end{equation}
Moreover the Pythagorean rule holds i.e., for any other distribution
$\pp \in \mathcal P$,
\begin{equation}\label{eq:p1}
\ii(\pp_{XY}||\qq_{XY})=\ii(\pp_{XY}||\pp^*_{XY})+\ii(\pp^*_{XY}||\qq_{XY}),
\end{equation}
and one also has
\begin{equation}\label{eq:pq0}
\ii(\pp^*_{XY}||\qq_{XY})=\ii(\pp^0_Y||\qq_Y).
\end{equation}
\end{proposition}
\begin{proof}
The starting point is equation~(\ref{eq:divsplit}), which now takes
the form
\begin{equation}\label{eq:divsplit0}
\ii(\pp_{XY}||\qq_{XY})=\ii(\pp^0_Y||\qq_Y)+\mathbb{E}_{\pp_Y}\ii(\pp_{X|Y}||\qq_{X|Y}).
\end{equation}
Since the first term on the right hand side is fixed, the minimizing
$\pp^*_{XY}$ must satisfy $\pp^*_{X|Y}=\qq_{X|Y}$. It follows that
$\pp^*_{XY}=\pp^*_{X|Y}\pp^0_Y=\qq_{X|Y}\pp^0_Y$, thus
verifying~(\ref{eq:rn}) and~(\ref{eq:pq0}). We finally show
that~(\ref{eq:p1}) holds.
\begin{align*}
\ii(\pp_{XY}||\qq_{XY}) & =\mathbb{E}_{\pp_{XY}} \log \frac{\dd\pp_{XY}}{\dd\pp^*_{XY}}+\mathbb{E}_{\pp_{XY}} \log \frac{\dd\pp^*_{XY}}{\dd\qq_{XY}} \\
& = \ii(\pp_{XY}||\pp^*_{XY}) + \mathbb{E}_{\pp_Y}\log \frac{\dd\pp^0_Y}{\dd\qq_Y} \\
& = \ii(\pp_{XY}||\pp^*_{XY}) + \mathbb{E}_{\pp^0_Y}\log
\frac{\dd\pp^0_Y}{\dd\qq_Y},
\end{align*}
where we used that any $\pp_{XY} \in \mathcal P$ has $Y$-marginal
distribution $\pp^0_Y$.
\end{proof}
The results above can be extended to the case where the random
vector $(X, Y):=(X, Y_1, \dots Y_m)$, i.e. $Y$ consists of $m$
random subvectors $Y_i$. For any probability distribution $\pp_{XY}$
on $(X, Y)$, consider the conditional distributions $\pp_{Y_i|X}$
and define the probability distribution $\widetilde{\pp}_{XY}$ on
$(X,Y)$:
$$
\widetilde{\pp}_{XY} = \prod_i \pp_{Y_i|X} \pp_X.
$$
Note that, under $\widetilde{\pp}_{XY}$, the $Y_i$ are conditionally
independent given $X$. The following lemma sharpens
Lemma~\ref{lemma:xcondy}.
\begin{lemma}\label{lemma:xyi}
Let $\pp_{XY}$ and $\qq_{XY}$ be given probability distributions of
a Euclidean random vector $(X,Y):=(X,Y_1,\dots Y_m)$. Assume that
$\pp_{XY} \ll \qq_{XY}$ and that, under $\qq_{XY}$, the subvectors
$Y_i$ of $Y$ are conditionally independent given $X$, then
\[
\ii(\pp_{XY}||\qq_{XY})=\ii(\pp_{XY}||\widetilde{\pp}_{XY})+\sum_i\mathbb{E}_{\pp_X}\ii(\pp_{Y_i|X}||\qq_{Y_i|X})
+ \ii(\pp_X||\qq_X).
\]
\end{lemma}
\begin{proof}
The proof runs along the same lines as the proof of
Lemma~\ref{lemma:xcondy}. We start from
equation~(\ref{eq:divsplit}) with the roles of $X$ and $Y$
reversed. With the aid of $\widetilde{\pp}_{XY}$ one can decompose
the term $\mathbb{E}_{\pp_X}\ii(\pp_{Y|X}||\qq_{Y|X})$ as follows.
\begin{align*}
\mathbb{E}_{\pp_X}\ii(\pp_{Y|X}||\qq_{Y|X}) & =
\mathbb{E}_{\pp_X}\mathbb{E}_{\pp_{Y|X}}\log\frac{\dd\pp_{Y|X}}{\dd\qq_{Y|X}} \\
& = \mathbb{E}_{\pp_X}\mathbb{E}_{\pp_{Y|X}}\left(\log\frac{\dd\pp_{Y|X}}{\dd\widetilde{\pp}_{Y|X}} +\log\frac{\dd\widetilde{\pp}_{Y|X}}{\dd\qq_{Y|X}}\right)\\
& = \mathbb{E}_{\pp_X}\ii(\pp_{Y|X}||\widetilde{\pp}_{Y|X}) + \mathbb{E}_{\pp_X}\mathbb{E}_{\pp_{Y|X}}\sum_i \log \frac{\dd\pp_{Y_i|X}}{\dd\qq_{Y_i|X}} \\
& = \ii(\pp_{XY}||\widetilde{\pp}_{XY}) + \sum_i \mathbb{E}_{\pp_X}\mathbb{E}_{\pp_{Y_i|X}} \log \frac{\dd\pp_{Y_i|X}}{\dd\qq_{Y_i|X}} \\
& = \ii(\pp_{XY}||\widetilde{\pp}_{XY})+\sum_i\mathbb{E}_{\pp_X}\ii(\pp_{Y_i|X}||\qq_{Y_i|X}),
\end{align*}
where we used the fact that
$\frac{\dd\pp_{XY}}{\dd\widetilde{\pp}_{XY}} =
\frac{\dd\pp_{Y|X}}{\dd\widetilde{\pp}_{Y|X}}$. This proves the
lemma.
\end{proof}
The decomposition of Lemma~\ref{lemma:xyi} is useful when solving
I-divergence minimization problems with conditional independence
constraints, like the one considered below.
\begin{proposition}\label{prop:div2min}
Let $\pp_{XY}$ be a given probability distribution of a Euclidean
random vector $(X,Y):=(X,Y_1,\dots Y_m)$. Consider the I-divergence
minimization problem
$$
\min_{\qq_{XY} \in \mathcal Q} \, \ii(\pp_{XY}||\qq_{XY}),$$ where
$$
\mathcal Q := \{ \qq_{XY} \,\, | \,\, \qq_{Y_1, \dots, Y_m|X} =
\prod_i \qq_{Y_i|X} \}.
$$
If $\pp_{XY} \ll \qq_{XY}$ for some $\qq_{XY} \in \mathcal Q$ then
the I-divergence is minimized by
$$\qq^*_{XY}=\widetilde{\pp}_{XY}$$
Moreover, the Pythagorean rule holds, i.e. for any $\qq_{XY} \in
\mathcal Q$,
\[
\ii(\pp_{XY}||\qq_{XY})= \ii(\pp_{XY}||\qq^*_{XY}) +
\ii(\qq^*_{XY}||\qq_{XY}).
\]
\end{proposition}
\begin{proof}
{From} the right hand side of the  identity in Lemma~\ref{lemma:xyi}
we see that the first I-divergence is not involved in the
minimization, whereas the other two can be made equal to zero, by
selecting $\qq_{Y_i|X}=\pp_{Y_i|X}$ and $\qq_X=\pp_X$. This shows
that the
minimizing $\qq^*_{XY}$ is equal to $\widetilde{\pp}_{XY}$. \\
To prove the Pythagorean rule, we first observe that trivially
\begin{equation}\label{eq:d1}
\ii(\pp_{XY}|\qq^*_{XY})=\ii(\pp_{XY}|\widetilde{\pp}_{XY}).
\end{equation}
Next we apply the identity in Lemma~\ref{lemma:xyi} with
$\qq^*_{XY}$ replacing $\pp_{XY}$. In this case the corresponding
$\widetilde{\qq}^*_{XY}$ obviously equals $\qq^*_{XY}$ itself. Hence
the identity reads
\begin{align}
\ii(\qq^*_{XY}||\qq_{XY})&=\sum_i\mathbb{E}_{\qq^*_X}\ii(\qq^*_{Y_i|X}||\qq_{Y_i|X})
+ \ii(\qq^*_X ||\qq_X) \nonumber\\
& = \sum_i\mathbb{E}_{\pp_X}\ii(\pp_{Y_i|X}||\qq_{Y_i|X}) +
\ii(\pp_X||\qq_X),\label{eq:d22}
\end{align}
by definition of $\qq^*_{XY}$. Adding up equations~(\ref{eq:d1})
and~(\ref{eq:d22}) gives the result.
\end{proof}

\section{Proof of the technical results}\label{section:tech}
\setcounter{equation}{0}

{\sc Proof of Proposition~\ref{prop:exist}.} \emph{Existence of the
minimum.} Let $(H_0, D_0)$ be arbitrary. Perform one iteration of
the algorithm to get $(H_1, D_1)$ with
$\ii(\widehat{\Sigma}||H_1H_1^\top +D_1)\leq
\ii(\widehat{\Sigma}||H_0H_0^\top +D_0)$. Moreover, from
Proposition~\ref{prop:properties}, $H_1H_1^\top \leq
\widehat{\Sigma}$ and $D_1\leq \Delta(\widehat{\Sigma})$. Hence the
search for a minimum can be restricted to the set of matrices
$(H,D)$ satisfying $HH^\top \leq \widehat{\Sigma}$ and $D\leq
\Delta(\widehat{\Sigma})$. We claim that the search for a minimum
can be further restricted to the set of $(H, D)$ such that $HH^\top
+D\geq \eps I$ for some sufficiently small $\eps>0$. Indeed, if the
last inequality is violated, then   $HH^\top +D$ has at least one
eigenvalue less than $\eps$. Assume this is the case, write $HH^\top
+ D=U\Lambda U^\top$, the Jordan decomposition of $HH^\top + D$, and
let $\widehat{\Sigma} = U \Sigma_U U^\top$. Then
$\ii(\widehat{\Sigma}||HH^\top + D)=\ii(\Sigma_U||\Lambda)$, as one
easily verifies. Denoting by $\lambda_i$ the eigenvalues of $HH^\top
+ D$, $\lambda_{i_0}$ the smallest among them, and by $\sigma_{ii}$
the diagonal elements of $\Sigma_U$, we have that $\ii(\Sigma_U|
\Lambda) = \half \sum_i \left(\log\lambda_i +
\frac{\sigma_{ii}}{\lambda_i}\right) - \half \log|\Sigma_U| -
\frac{n}{2}$. Choose $\eps$ smaller than $c := \min_i \sigma_{ii}
>0$, since $\widehat \Sigma >0$. Then the contribution of $i=i_0$ in
the summation is larger than $\log \eps + \frac{c}{\eps}$ which
tends to infinity for $\eps \to 0$. Hence the claim is verified.
This shows that a minimizing pair $(H, D)$ has to satisfy
$HH^\top\leq \widehat{\Sigma}$, $D\leq \Delta(\widehat{\Sigma})$,
and $HH^\top + D\geq \eps I$, for some $\eps > 0$. In other words we
have to minimize the I-divergence over a compact set on which it is
clearly continuous. This proves Proposition~\ref{prop:exist}.
\mbox{}~\hfill$\square$
\medskip

{\sc Proof of Proposition~\ref{prop:pqq}.} \emph{Relation between
the original and the lifted problem.} Let
$\Sigma_1=\Sigma(H,D,Q)$. With $\Sigma^*=\Sigma^*(\Sigma_1)$, the
optimal solution of the partial minimization over $\bs_0$, we have
for any $\Sigma_0\in\bs_0$, using~(\ref{eq:011}) in the first
equality below,
\begin{align*}
\ii(\Sigma_0||\Sigma_1)& \geq \ii(\Sigma^*||\Sigma_1) \\
& = \ii(\widehat{\Sigma}||HH^\top + D) \\
& \geq  \inf_{H,D}\ii(\widehat{\Sigma}||HH^\top + D).
\end{align*}
It follows that
$\inf_{\Sigma_0\in\bs_0,\Sigma_1\in\bs_1}\ii(\Sigma_0 ||
\Sigma_1)\geq\min_{H,D}\ii(\widehat{\Sigma}||HH^\top + D)$, since
the minimum exists in view of Proposition~\ref{prop:exist}.
\\
Conversely, let $(H^*,D^*)$ be the minimizer of $(H,D)\mapsto
\ii(\widehat{\Sigma}||HH^\top+D)$, which exists by
Proposition~\ref{prop:exist}, and let $\Sigma^*=\Sigma(H^*,D^*,Q^*)$
be a corresponding element in $\bs_1$. Furthermore, let
$\Sigma^{**}\in\bs_0$ be the minimizer of $\Sigma_0\mapsto
\ii(\Sigma_0||\Sigma^*)$ over $\bs_0$. Then we have
\begin{align*}
\ii(\widehat{\Sigma}||H^*H^{*^\top}+D^*) & = \ii(\Sigma^{**}|| \Sigma^*)\\
& \geq
\inf_{\Sigma_0\in\bs_0,\Sigma_1\in\bs_1}\ii(\Sigma_0||\Sigma_1),
\end{align*}
which shows the other inequality. Finally, we prove that the the
infimum can be replaced with a minimum. Thereto we will explicitly
construct a minimizer in terms of $(H^*,D^*)$.  For any invertible
$Q^*$ let $\Sigma^*=\Sigma(H^*,D^*,Q^*)$. Performing the first
partial minimization, we obtain an optimal $\Sigma^{**}\in\bs_0$,
with the property (see (\ref{eq:011})) that $\ii(\Sigma^{**}|
\Sigma^*)=\ii(\widehat{\Sigma}||H^*H^{*^\top}+D^*)$.
\mbox{}\hfill$\square$
\medskip

{\sc Proof of Proposition~\ref{prop:pm1}.} \emph{First partial
minimization.} Consider the setup and the notation of
Proposition~\ref{prop:gpm1}. Identify $\qq$ with the normal
$N(0,\Sigma)$, and $\pp$ with $N(0,\Sigma_0)$. By virtue
of~(\ref{eq:rn}), the optimal $\pp^*$ is a zero mean normal whose
covariance matrix can be computed using the properties of
conditional normal distributions (see appendix~\ref{gauss}). In
particular
\begin{align*}
\Sigma^*_{21}  = \mathbb{E}_{\pp^*}XY^\top  & =
 \mathbb{E}_{\pp^*}(\mathbb{E}_{\pp^*}[X|Y]Y^\top)
\\
& =\mathbb{E}_{\pp^*}(\mathbb{E}_\qq[X|Y]Y^\top)
\\
& =\mathbb{E}_{\pp^*}(\Sigma_{21}\Sigma_{11}^{-1}YY^\top) \\
& = \Sigma_{21}\Sigma_{11}^{-1}\mathbb{E}_{\pp^0}YY^\top
\\
& =\Sigma_{21}\Sigma_{11}^{-1}\widehat{\Sigma}.
\end{align*}
Likewise
\begin{align*}\Sigma^*_{22}  = \mathbb{E}_{\pp^*}XX^\top & = \cov_{\pp^*}(X|Y) + \mathbb{E}_{\pp^*}
(\mathbb{E}_{\pp^*}[X|Y]\mathbb{E}_{\pp^*}[X|Y]^\top) \\
& = \cov_{\qq}(X|Y) + \mathbb{E}_{\pp^*}
(\mathbb{E}_{\qq}[X|Y]\mathbb{E}_{\qq}[X|Y]^\top) \\
& = \Sigma_{22}-\Sigma_{21}\Sigma_{11}^{-1}\Sigma_{12} +
\mathbb{E}_{\pp^*}
(\Sigma_{21}\Sigma_{11}^{-1}Y(\Sigma_{21}\Sigma_{11}^{-1}Y)^\top) \\
& = \Sigma_{22}-\Sigma_{21}\Sigma_{11}^{-1}\Sigma_{12} +
\mathbb{E}_{\pp^0}
(\Sigma_{21}\Sigma_{11}^{-1}YY^\top\Sigma_{11}^{-1}\Sigma_{12}) \\
& = \Sigma_{22}-\Sigma_{21}\Sigma_{11}^{-1}\Sigma_{12} +
\Sigma_{21}\Sigma_{11}^{-1}\widehat{\Sigma}\Sigma_{11}^{-1}\Sigma_{12}.
\end{align*}
To prove that $\Sigma^*$ is strictly positive note first that
$\Sigma^*_{11} = \widehat \Sigma >0$ by assumption. To conclude,
since $\Sigma >0$, it is enough to note that
\[
\Sigma^*_{22}-\Sigma^*_{21} (\Sigma^*_{11})^{-1}\Sigma^*_{12} \, =
\, \Sigma_{22}-\Sigma_{21}\Sigma_{11}^{-1}\Sigma_{12}
\]
Finally, the relation
$\ii(\Sigma^*||\Sigma)=\ii(\widehat{\Sigma}||\Sigma_{11})$ is
Equation~(\ref{eq:pq0}) adapted to the present situation. The
Pythagorean rule follows from this relation and
Equation~(\ref{eq:divsplit0}). \mbox{}~\hfill$\square$
\medskip

{\sc Proof of Proposition~\ref{prop:pm2}.} \emph{Second partial
minimization.} We adhere to the setting and the notation of
Proposition~\ref{prop:div2min}. Identify $\pp=\pp_{XY}$ with the
normal distribution $N(0, \Sigma)$ and $\qq=\qq_{XY}$ with the
normal $N(0, \Sigma_1)$, where $\Sigma_1\in\bs_1$. The optimal
$\qq^*=\qq^*_{XY}$ is again normal and specified by its
(conditional) mean and covariance matrix. Since
$\qq^*_{Y_i|X}=\pp_{Y_i|X}$ for all $i$, we have
$\mathbb{E}_{\qq^*}[Y|X]=\mathbb{E}_{\pp}[Y|X]=\Sigma_{12}\Sigma_{22}^{-1}X$,
moreover $\qq^*_X = \pp_X$. Hence we find
$$\Sigma^*_{12} = \mathbb{E}_{\qq^*} YX^\top =
\mathbb{E}_{\qq^*}\mathbb{E}_{\qq^*}[Y|X]X^\top
=\mathbb{E}_{\pp}\mathbb{E}_{\pp}[Y|X]X^\top =\Sigma_{12}.$$
Furthermore, under $\qq^*$, the $Y_i$ are conditionally
independent given $X$. Hence $\cov_{\qq^*}(Y_i,Y_j|X)=0$, for
$i\neq j$, whereas $\mathbb{V}{\rm
ar}_{\qq^*}(Y_i|X)=\mathbb{V}{\rm ar}_{\pp}(Y_i|X)$, which is the
${ii}$-element of
$(\Sigma_{11}-\Sigma_{12}\Sigma_{22}^{-1}\Sigma_{21})$, it follows
that
$$\cov_{\qq^*}(Y|X)=\Delta(\Sigma_{11}-\Sigma_{12}\Sigma_{22}^{-1}\Sigma_{21}).$$
We can now evaluate
\begin{align*}
\Sigma^*_{11} = \cov_{\qq^*}(Y) & = \mathbb{E}_{\qq^*}YY^\top \\
& = \mathbb{E}_{\qq^*}( \ee_{\qq^*}[Y|X]\ee[Y|X]^\top + \cov_{\qq^*}(Y|X) ) \\
& =\mathbb{E}_{\qq^*}(\Sigma_{12}\Sigma_{22}^{-1}XX^\top
\Sigma_{22}^{-1}\Sigma_{21} +
\Delta(\Sigma_{11}-\Sigma_{12}\Sigma_{22}^{-1}\Sigma_{21}) ) \\
& =\Sigma_{12}\Sigma_{22}^{-1}\Sigma_{21} +
\Delta(\Sigma_{11}-\Sigma_{12}\Sigma_{22}^{-1}\Sigma_{21}).
\end{align*}
The Pythagorean rule follows from the general result of
Proposition~\ref{prop:div2min}. \hfill$\square$
\medskip

{\sc Proof of Proposition~\ref{prop:sigma0}.} \emph{Constrained
second partial minimization.} Lemma~\ref{lemma:xyi} and
Proposition~\ref{prop:div2min} still apply, with the proviso that
the marginal distribution of $X$ is fixed at some $\qq^0_X$. The
optimal distribution $\qq^*_{XY}$ will therefore take the form
$\qq^*_{XY}= \prod_i \pp_{Y_i|X} \qq^0_X$.
Turning to the explicit computation of the optimal normal law,
inspection of the proof of Proposition~\ref{prop:pm2} reveals that
under $\qq^*$ we have $\mathbb{E}_{\qq^*} YX^\top =
\Sigma_{12}\Sigma_{22}^{-1}P_0$ and
\[
\cov_{\qq^*}(Y)=\Delta(\Sigma_{11}-\Sigma_{12}\Sigma_{22}^{-1}\Sigma_{21})+
\Sigma_{12}\Sigma_{22}^{-1}P_0\Sigma_{22}^{-1}\Sigma_{21}.
\]
\hfill $\square$

\medskip
{\sc Proof of Lemma~\ref{lemma:divdecomposition}.} \emph{Technical
decomposition of the I-divergence.} Recall the following notation.
\begin{align}
H = & \begin{pmatrix}
H_1 \\
H_2
\end{pmatrix} \\
D = & \begin{pmatrix}
D_1 & 0 \\
0 & D_2
\end{pmatrix}, \\
\widehat \Sigma = & \begin{pmatrix} \widehat \Sigma_{11} & \widehat \Sigma_{12} \\
\widehat \Sigma_{21} & \widehat \Sigma_{22}
\end{pmatrix}, \\
{\widetilde \Sigma}_{11} = & \widehat \Sigma_{11} - \widehat
\Sigma_{12} \widehat \Sigma_{22}^{-1} \widehat \Sigma_{21},
\end{align}
where $H_1\in\R^{n_1\times k}$, $H_2\in\R^{n_2\times k}$,
$D_1\in\R^{n_1\times n_1}$ and $D_2\in\R^{n_2\times n_2}$.
\medskip\\
Define $S=H_1(I-H_2^\top(H_2H_2^\top + D_2)^{-1}H_2)H_1^\top+D_1$
and
$K=\widehat{\Sigma}_{12}\widehat{\Sigma}_{22}^{-1}-H_1H_2^\top(H_2H_2^\top+D_2)^{-1}$.
From Lemma~\ref{lemma:xcondy} we obtain that
$\ii(\widehat{\Sigma}||HH^\top + D)$ is the sum of
$\ii(\Sigma_{22}||H_2H_2^\top +D_2)$ and an expected I-divergence
between conditional distributions. The latter can be computed
according to Equation~(\ref{eq:divmu}), and gives the decomposition
result
\begin{equation}\label{eq:ddt}
\ii(\widehat{\Sigma}||HH^\top + D) =
\ii(\widehat{\Sigma}_{22}||H_2H_2^\top +D_2) +
\ii(\widetilde{\Sigma}_{11}||S)  + \half{\rm
tr}\{S^{-1}K\widehat{\Sigma}_{22}K^\top\}.
\end{equation}
The assertion of Lemma~\ref{lemma:divdecomposition} is then obtained
by taking $D_2=0$ and the further decomposition of $H_1$ and $H_2$
as in~(\ref{eq:h21}). \hfill$\square$

\bigskip
\bibliography{lista_FA}
\bibliographystyle{imsart-nameyear}


\end{document}